\documentclass[11pt]{article}
\usepackage{amsmath,amssymb,amsfonts,amsthm,amscd,mathrsfs}
\usepackage{latexsym, euscript, epic, eepic}
\usepackage{graphicx}
\usepackage{verbatim}
\usepackage{fourier}
\usepackage{tikz}

\usepackage[colorlinks=true,linkcolor=black,citecolor=black,urlcolor=black]{hyperref}

\makeatletter
\newcommand{\subsectionruninhead}{\@startsection{subsection}{2}{0mm}
{-\baselineskip}{-0mm}{\bf\large}}
\newcommand{\subsubsectionruninhead}{\@startsection{subsubsection}{3}{0mm}
{-\baselineskip}{-0mm}{\bf\normalsize}}
\makeatother

\newtheorem*{theorem*}{Theorem}

\newtheorem{theoremalph}{Theorem}

\newtheorem*{proposition*}{Proposition}
\newtheorem*{corollary*}{Corollary}
\newtheorem*{claim*}{Claim}
\newtheorem*{remark*}{Remark}
\newtheorem*{problem*}{Problem}
\newtheorem{theorem}{Theorem}[section]

\newtheorem{proposition}[theorem]{Proposition}
\newtheorem{corollary}[theorem]{Corollary}
\newtheorem{lemma}[theorem]{Lemma}

\theoremstyle{definition}
\newtheorem{definition}[theorem]{Definition}
\newtheorem{remark}[theorem]{Remark}

\numberwithin{equation}{section}

\newcommand{\htop}{h_{\text{top}}}

\begin{document}

\title{Analytic Paths of Ergodic Measures with Prescribed Invariants}

\author{{Xiaobo Hou$^{1,\S}$ and Xueting Tian$^{2,\dag}$}\\
		{\em\small$^1$ School of Mathematical Science,   Dalian University of Technology}\\
	{\em\small Dalian 116024, People's Republic of China}\\
	{\em\small$^2$ School of Mathematical Science,  Fudan University}\\
	{\em\small Shanghai 200433, People's Republic of China}\\
	{\small $^\S$Email: xiaobohou@dlut.edu.cn;   $^{\dag}$Email: xuetingtian@fudan.edu.cn}\\}

\date{}



\maketitle

\begin{abstract}
	We study analytic paths of ergodic measures under quantitative constraints.
	For uniformly hyperbolic systems, we construct one-parameter families of
	ergodic measures whose prescribed Birkhoff averages vary affinely, whose
	metric entropies vary analytically, and whose endpoint entropy values are
	realized exactly. Along these paths, the integral of every H\"older continuous
	observable depends analytically on the parameter. In the mixing case, the
	measures may be chosen to be Bernoulli.
	
	We also prove a dimension counterpart for average conformal hyperbolic sets:
	the Hausdorff dimensions of the measures vary analytically along the path.
	
	Finally, for a class of partially hyperbolic diffeomorphisms with
	one-dimensional center, we construct analytic paths of ergodic measures whose
	center Lyapunov exponents are prescribed linearly and whose entropies vary
	analytically.
\end{abstract}

\tableofcontents

\section{Introduction}

\qquad A basic problem in dynamics is to understand how ergodic measures realize quantitative invariants such as metric entropy, Birkhoff averages, and Lyapunov exponents. In many settings one can prescribe one or more of these quantities and construct  ergodic measures realizing them. A natural next question is whether such realizations can be organized into analytic one-parameter families: 
\begin{quote}
	\itshape
	Can quantitative invariants be realized along analytic paths of ergodic measures?
\end{quote}

This question is naturally related to the \emph{flexibility program} in dynamics, initiated by Katok, which asks whether a given dynamical quantity can attain arbitrary values within a prescribed class of systems. One aspect of this program concerns entropy: starting from Katok's seminal work on horseshoes of large entropy, subsequent works have shown that ergodic measures realize arbitrary intermediate entropy values in many settings \cite{Katok1980,GuanSunWu2017,HuangXuXu2021,LiOprocha2018,LiShiWangWang2020,Sun2025,Burguet2020,ChandgotiaMeyerovitch2021,QuasSoo2016,Ures2012,YangZhang2020}. A second aspect concerns Lyapunov exponents: Bochi et al.\ made this perspective explicit for volume-preserving diffeomorphisms, and  work of B\'ar\'any et al., D\'iaz et al.,  and Tian et al.\ showed that ergodic measures can realize nontrivial ranges of Lyapunov exponents in several settings \cite{BochiKatokRodriguezHertz2021,DGR2019,BJKR2021,TianWangWang2019}. A third aspect concerns joint constraints, namely the simultaneous prescription of entropy with Birkhoff averages or Lyapunov exponents. In this direction, Dong et al. developed multi-horseshoe methods for entropy, Birkhoff average, and Lyapunov exponent constraints in hyperbolic, singular hyperbolic, and certain nonhyperbolic settings \cite{DHT,HouTian2023}. Building on this line of work, D\'iaz et al.\ established continuous paths of ergodic measures with fixed center exponent for certain partially hyperbolic diffeomorphisms with one-dimensional center \cite{DGMZ2025}. In contrast to these existence and continuous path results, we study whether such realizations can be organized into analytic paths of ergodic measures.

A further motivation comes from the connectedness theory of ergodic measures. Sigmund proved that the set of ergodic measures is connected for mixing subshifts of finite type, and hence for basic hyperbolic sets via Markov partitions \cite{Sigmund1977}. Later, Konieczny et al. obtained related results in the $\bar d$-topology for symbolic systems with a safe symbol and linked them to intermediate entropy phenomena \cite{Konieczny2018}. More recently, Sun proved  connectedness for ergodic invariant measures of systems with the approximate product property \cite{Sun2025}. These works motivate a quantitative refinement of connectedness: rather than merely connecting ergodic measures, one seeks connecting paths along which averages, exponents, entropy, or dimension are prescribed. Our goal is to realize such paths with analytic regularity.

The problem is also connected with multifractal analysis. For fixed potentials, the level sets of Birkhoff averages or Lyapunov exponents define natural subsets of the phase space, and a classical question is to determine their entropy or dimension; see, for instance, \cite{Fan2021,HeWolf2022,BJKR2021,Climenhaga2013,JohanssonJordanObergPollicott2010,Jenkinson2001,DGR2019,IJ2015,JT2021,JordanRams2021,PS2007,LiuLiu2025,BS2001,BarreiraHolanda2021,KucherenkoWolf2014,PesinSadovskaya2001}.  In the present paper, these multifractal level sets serve as the constraint classes inside which we construct analytic paths of ergodic measures with prescribed entropy, dimension, or Lyapunov exponents.

The goal of the present article is to construct analytic paths of ergodic measures under prescribed quantitative constraints. While existing intermediate value results typically guarantee the existence of a single ergodic measure realizing a given entropy, dimension, or Lyapunov exponent, we investigate whether such realizations can be organized into analytic families. Along these paths, Birkhoff averages or Lyapunov exponents are fixed pointwise, while entropy or dimension varies analytically with the parameter. This leads to a notion of constrained analytic connectedness, rather than ordinary connectedness in the space of ergodic measures.

In uniformly hyperbolic systems, we construct families $\{\mu_t\}_{t\in[0,1]}$ of ergodic measures for which the prescribed Birkhoff averages move linearly between two given values, while the metric entropy is analytic and realizes the prescribed endpoint values exactly. The path is analytic in the sense that, for every H\"older
continuous observable $\psi$, the map
$$
t\mapsto \int \psi\,d\mu_t
$$
is analytic. In the mixing case, the measures can be chosen among Bernoulli measures. We also obtain a dimension counterpart for average conformal hyperbolic sets, where the Hausdorff dimension of the measures varies analytically along the path. 

In addition, for a class of partially hyperbolic diffeomorphisms with one dimensional center, we construct analytic paths of ergodic measures whose center Lyapunov exponents are prescribed linearly and whose entropies vary analytically.

Throughout the paper, a dynamical system $(X,f)$ means that $(X,\rho)$ is a compact metric space and $f:X\to X$ is continuous.
We denote by $\mathcal{M}(X)$, $\mathcal{M}_f(X)$ and $\mathcal{M}_{f}^{e}(X)$ the sets of probability measures, $f$-invariant probability measures and $f$-ergodic probability measures, respectively. The spaces of continuous functions on $X$ and H\"older continuous functions of exponent $\gamma$ on $X$ are denoted by $C(X)$ and $C^\gamma(X)$, respectively.
Given a positive integer $d$, let $\Phi=\left(\varphi_1, \ldots, \varphi_d\right)\in C(X)^d$.
For $\alpha=\left(\alpha_1, \ldots, \alpha_d\right)\in\mathbb{R}^d$, set
$$
K_\alpha=K_\alpha(\Phi;f)=\bigcap_{i=1}^d\left\{x \in X: \lim _{n \rightarrow \infty} \frac{\sum_{k=0}^{n-1} \varphi_i\left(f^k x\right)}{n}=\alpha_i\right\}.
$$
Define
$$
\mathcal{P}(\mu)=\left(\int_X \varphi_1 \mathrm{~d} \mu, \ldots, \int_X \varphi_d \mathrm{~d} \mu\right) .
$$
For $1\leq i\leq d$, write $\mathcal{P}_i(\mu)=\int_X \varphi_i\,\mathrm{d}\mu$. Since $\mathcal{M}_f(X)$ is compact and convex and $\mathcal{P}$ is continuous, the set $\mathcal{P}(\mathcal{M}_f(X))$ is compact and convex. We denote its interior by $\operatorname{Int}\mathcal{P}(\mathcal{M}_f(X))$. The support of a measure $\mu$ is denoted by $S_\mu$, and the set of Bernoulli measures by $\mathcal{M}_f^B(X)$.

Our first result concerns uniformly hyperbolic systems and provides analytic paths realizing both prescribed averages and prescribed entropies. For later reference, we label the following three classes of systems by $(S1)$, $(S2)$ and $(S3)$.
\begin{enumerate}
	\item[(S1)] $f:X\to X$ is a subshift of finite type;
	\item[(S2)] $X$ is a locally maximal hyperbolic set of a $C^{1+\varepsilon}$ diffeomorphism $f$ for some $\varepsilon>0$;
	\item[(S3)] $X$ is a repeller  of a $C^{1+\varepsilon}$ map $f$ for some $\varepsilon>0$.
\end{enumerate}

\begin{theoremalph}\label{Thm-int}
	Let $(X,f)$ be a dynamical system satisfying one of the conditions $(S1)$, $(S2)$ and $(S3)$.
	Assume that $(X,f)$ is topologically 
	 mixing (resp. topologically  transitive). Given $d\in\mathbb{N}^+,$ $0<\gamma\leq 1$ and a vector $\Phi \in C^\gamma(X)^d$.  Then for any $\hat{\alpha},\tilde{\alpha}\in \operatorname{Int}\mathcal{P}(\mathcal{M}_f(X))$, $0< \hat{h}\leq h_{top}(f, K_{\hat{\alpha}})$ and $0<  \tilde{h}\leq  h_{top}(f, K_{\tilde{\alpha}}),$  there is a  path $$\{\mu_t\}_{t\in[0,1]}\subset \mathcal{M}_f^B(X)\qquad \left(\text{resp.}  ~\{\mu_t\}_{t\in[0,1]}\subset \mathcal{M}_f^e(X)\right)$$ such that 
	\begin{enumerate}
		\item the map $t\mapsto \int \psi d\mu_t$ is  analytic for any H\"older continuous function $\psi$;
		\item the map $t\to h_{\mu_t}(f)$ is analytic and $\inf_{t\in [0,1]}h_{\mu_t}(f)=\min\{\hat{h},\tilde{h}\}$;
		\item $(\mathcal{P}(\mu_0),h_{\mu_0}(f))=(\tilde{\alpha},\tilde{h})\text{ and }(\mathcal{P}(\mu_1),h_{\mu_1}(f))=(\hat{\alpha},\hat{h})$;
		\item    $\mathcal{P}_i(\mu_t)=t\hat{\alpha}_i+(1-t)\tilde{\alpha}_i$ for any $1\leq i \leq d$ and $t\in[0,1]$. In particular, if $\hat{\alpha}=\tilde{\alpha}=\alpha,$ then $\mathcal{P}(\mu_t)=\alpha$ for any $t\in[0,1]$;
		\item $S_{\mu_t}=X$ for any $t\in[0,1].$ 
	\end{enumerate}
\end{theoremalph}

Besides entropy, our method also yields a dimension counterpart in the average conformal setting, where conformality is needed to identify the Hausdorff dimension of measures along the path.
We say that a hyperbolic set $X$ is average conformal if, for every
$\mu\in \mathcal{M}_f^e(X)$, all Lyapunov exponents of $\mu$ along
the stable direction coincide and all Lyapunov exponents of $\mu$
along the unstable direction coincide. Equivalently, every
$\mu\in \mathcal{M}_f^e(X)$ has exactly two distinct Lyapunov
exponents, one negative and one positive.
For any $\alpha\in\mathcal{P}(\mathcal{M}_f(X)),$ denote 
$$\dim_H (\alpha):=\sup\{\dim_H \mu: \mu\in\mathcal{M}_f^e(X) \text{ and } \mathcal{P}(\mu)=\alpha\}.$$
\begin{theoremalph}\label{thm-huasdorff}
	Let $X$ be a locally maximal average conformal hyperbolic set of a $C^{1+\varepsilon}$ diffeomorphism $f$.
	Assume that $(X,f)$ is topologically mixing
	(resp. topologically transitive). Given $d\in\mathbb{N}^+$,
	$0<\gamma\leq 1$, and
	$\Phi\in C^\gamma(X)^d$,
	then for any
	$
	\hat{\alpha},\tilde{\alpha}
	\in \operatorname{Int}\mathcal{P}(\mathcal{M}_f(X)),
	$
	and any
	$
	0<\hat h<\dim_H(\hat\alpha),
	~
	0<\tilde h<\dim_H(\tilde\alpha),
	$
	there is a path
	$$
	\{\mu_t\}_{t\in[0,1]}\subset \mathcal{M}_f^B(X)
	\qquad
	\left(
	\mathrm{resp.}\ 
	\{\mu_t\}_{t\in[0,1]}\subset \mathcal{M}_f^e(X)
	\right)
	$$
	such that
	\begin{enumerate}
		\item the map $t\mapsto \int \psi d\mu_t$ is  analytic for any H\"older continuous function $\psi$;
		
		\item the map
		$
		t\mapsto \dim_H\mu_t
		$
		is analytic; 
		
		\item
		$
		(\mathcal{P}(\mu_0),\dim_H\mu_0)
		=
		(\tilde{\alpha},\tilde h),
		~
		(\mathcal{P}(\mu_1),\dim_H\mu_1)
		=
		(\hat{\alpha},\hat h);
		$
		
		\item for every $1\leq i\leq d$ and every $t\in[0,1]$,
		$
		\mathcal{P}_i(\mu_t)
		=
		t\hat{\alpha}_i+(1-t)\tilde{\alpha}_i.
		$
		In particular, if $\hat{\alpha}=\tilde{\alpha}=\alpha$, then
		$
		\mathcal{P}(\mu_t)=\alpha
		$
		for every $t\in[0,1]$;
		
		\item $S_{\mu_t}=X$ for every $t\in[0,1]$.
	\end{enumerate}
\end{theoremalph}

We also establish a corresponding result for partially hyperbolic diffeomorphisms with one-dimensional center. 
Let $M$ be a compact Riemannian manifold without boundary. A diffeomorphism $f:M\to M$ is partially hyperbolic, if there exists a $D f$-invariant dominated splitting
$
T M=E^{\mathrm{ss}} \oplus E^{\mathrm{c}} \oplus E^{\mathrm{uu}}
$
such that $E^{\text {ss }}$ (resp. $E^{\text {uu }}$ ) is uniformly contracting (resp. expanding). For any $\varepsilon>0$, denote by $\operatorname{Diff }^{1+\varepsilon}(M)$ the space of $C^{1+\varepsilon}$ diffeomorphisms on $M$,  and by $\mathrm{PH}_{\mathrm{c}=1}^{1+\varepsilon}(M)$ the set of $C^{1+\varepsilon}$ partially hyperbolic diffeomorphisms on $M$ with a one-dimensional center bundle $E^{\mathrm{c}}$.
Let $\varepsilon>0$ and $f\in \mathrm{PH}_{\mathrm{c}=1}^{1+\varepsilon}(M)$. For an ergodic measure $\mu\in \mathcal{M}_f^e(M)$, we denote by
$$
\chi^c(\mu):=\int \log \|Df|_{E^c}\|\, d\mu
$$
its center Lyapunov exponent, and write
$$
\mathcal{M}_{\mathrm{erg},\chi}(f):=\{\mu\in \mathcal{M}_f^e(M):\chi^c(\mu)=\chi\}.
$$
Set
$$
\chi_{\min}
:=
\inf\{\chi^c(\mu):\mu\in\mathcal M_f^e(M)\},
\qquad
\chi_{\max}
:=
\sup\{\chi^c(\mu):\mu\in\mathcal M_f^e(M)\}.
$$
For $\chi\in (\chi_{\min},\chi_{\max})$, define
$$
H(\chi):=\sup\{h_\mu(f):\mu\in \mathcal{M}_{\mathrm{erg},\chi}(f)\}.
$$
For the partially hyperbolic systems with blender-horseshoes and minimal foliations considered in \cite{DGMZ2025} where many examples are discussed, we obtain analytic paths of ergodic measures with prescribed center Lyapunov exponents and entropies.
\begin{theoremalph}\label{thm:same-sign-path}
	Let $f\in \mathrm{PH}_{\mathrm{c}=1}^{1+\varepsilon}(M)$ satisfy:
	\begin{enumerate}
		\item $f$ has a stable
		 and  an unstable blender-horseshoe;
		\item the foliations $W^{\mathrm{ss}}$ and $W^{\mathrm{uu}}$ are both minimal.
	\end{enumerate}
	Let
	$
	\hat{\chi},\tilde{\chi}\in (\chi_{\min},\chi_{\max})$ with $\hat{\chi}\tilde{\chi}>0,
	$
	and let
	$
	0<\hat{h}<H(\hat{\chi}),$ $0<\tilde{h}<H(\tilde{\chi}).
	$
	Then there exists a path
	$$
	\{\mu_t\}_{t\in [0,1]}\subset \mathcal{M}_f^e(M)
	$$
	such that
	\begin{enumerate}
		\item the map
		$
		t\mapsto \int \psi\, d\mu_t
		$
		is analytic for any H\"older continuous function $\psi:M\to \mathbb{R}$;
		\item the map
		$
		t\mapsto h_{\mu_t}(f)
		$
		is analytic and
		$
		\inf_{t\in [0,1]} h_{\mu_t}(f)=\min\{\hat{h},\tilde{h}\};
		$
		\item
		$
		(\chi^c(\mu_0),h_{\mu_0}(f))=(\tilde{\chi},\tilde{h}),
		$ and
		$
		(\chi^c(\mu_1),h_{\mu_1}(f))=(\hat{\chi},\hat{h});
		$
		\item
		$
		\chi^c(\mu_t)=t\hat{\chi}+(1-t)\tilde{\chi} 
		$ for any $t\in[0,1]$.  In particular, if $\hat{\chi}=\tilde{\chi},$ then 	$\{\mu_t\}_{t\in [0,1]}\subset \mathcal{M}_{\mathrm{erg},\hat{\chi}}(f).$
	\end{enumerate}
\end{theoremalph}

Theorem \ref{thm:same-sign-path} is closely related to the conditional realization results of \cite{DHT,HouTian2023}, where multi-horseshoe methods are used to realize entropy, dimension, Birkhoff average, and Lyapunov exponent constraints in hyperbolic, singular hyperbolic, and certain nonhyperbolic settings. Building on this direction, \cite{DGMZ2025} treats $C^1$ partially hyperbolic diffeomorphisms with one-dimensional center under blender-horseshoe and minimal foliation assumptions, proves a restricted variational principle for center Lyapunov level sets, and constructs continuous paths of ergodic measures inside fixed center-exponent classes. Its proof combines cascades of horseshoes and Feldman-Katok convergence tools. Theorem \ref{thm:same-sign-path} is different in nature: under the stronger $C^{1+\varepsilon}$ regularity assumption, for two prescribed same-sign center exponents and admissible endpoint entropies, it produces an analytic path of ergodic measures whose center exponent moves linearly between the two endpoint values, whose entropy varies analytically, and whose endpoint exponent--entropy pairs are realized exactly.

\subsection*{Methods of proof}

The proofs are based on thermodynamic formalism together with a lifting argument that transforms the prescribed values problem into a higher dimensional parameter problem. In the uniformly hyperbolic settings, besides the original functions, we introduce an auxiliary coordinate and pass from $\Phi=(\varphi_1,\ldots,\varphi_d)$ to an enlarged family $\Phi^{new}=(\varphi_1,\ldots,\varphi_d,\varphi_{d+1})$. The extra observable $\varphi_{d+1}$ is chosen so that the new boundary slices correspond to measures of zero entropy, while the maximal value along the new level recovers the entropy or dimension of the original level set. In this way, a prescribed entropy or dimension is realized by selecting suitable interior values of the new coordinate. Once the problem is lifted to the interior of this higher dimensional rotation set, analytic dependence of pressure and equilibrium states yields analytic families of ergodic measures, and hence analytic variation of averages, entropy, and, in the average conformal case, Hausdorff dimension.

In the partially hyperbolic setting, we combine this analytic pathwise mechanism with hyperbolic approximation and transition tools available in the blender-horseshoe setting. This places the result in the same framework as \cite{DGMZ2025}, while the conclusion obtained here is analytic and pathwise: the center Lyapunov exponent is prescribed along the whole path, and the entropy varies analytically with prescribed endpoint values.

Thus the novelty of the paper lies not only in realizing prescribed dynamical values, but in realizing them through analytically parametrized families of ergodic measures. The constructions turn entropy, dimension, and center Lyapunov exponent into controlled functions along a single path of measures, rather than values attached to unrelated measures. In this sense, the results refine the usual connectedness and flexibility viewpoints by producing analytic paths with quantitative control along the path and exact realization of the endpoint values.

\subsection*{Outline of the paper}

Section \ref{sec-Preliminaries} recalls the preliminaries used later. Section \ref{sec-hyper-entropy} proves the main analytic path result in uniformly hyperbolic systems, first in the $v$-dimensional form Theorem \ref{Thm-int-2}, and then derives Theorem \ref{Thm-int} as the entropy case. Section \ref{sec-hyper-dimen} applies this result, together with the dimension formula for average conformal systems, to prove Theorem \ref{thm-huasdorff}. Section \ref{sec-hyper-partial} treats the partially hyperbolic case: using the blender-horseshoe assumptions, it reduces the problem to a suitable hyperbolic set and proves Theorem \ref{thm:same-sign-path}.

\section{Preliminaries}\label{sec-Preliminaries}

\subsection{Bernoulli measures and periodic measures}
Let $Y=\{0,1,\ldots,k-1\}$ be a finite set,
$\mathcal{F}=2^Y$, and let
$
(p_0,p_1,\ldots,p_{k-1})
$
be a probability vector such that
$
p_i>0,~ 0\leq i\leq k-1,
$
and
$
\sum_{i=0}^{k-1}p_i=1.
$
Define a probability measure $\nu$ on $(Y,\mathcal{F})$ by
$
\nu(A)=\sum_{i\in A}p_i
$
for every $A\subset Y$. In particular,
$
\nu(\{i\})=p_i,~ 0\leq i\leq k-1.
$
Then $(Y,\mathcal{F},\nu)$ is a probability space.

Let
$
I=\mathbb{N}_0
$
or
$
I=\mathbb{Z},
$
where
$
\mathbb{N}_0=\{0,1,2,\ldots\}.
$
Write
$
(\Omega_I,\mathcal{B}_I,m_I)=\prod_{i\in I}(Y,\mathcal{F},\nu).
$
This product probability space is defined as follows. Let
$
\Omega_I=\prod_{i\in I}Y.
$
For every finite subset $\Lambda\subset I$ and every family
$
(a_i)_{i\in\Lambda}\in Y^\Lambda,
$
define the cylinder set
$
[a_i]_{i\in\Lambda}
=
\{(x_i)_{i\in I}\in\Omega_I: x_i=a_i,\ i\in\Lambda\}.
$
Let $\mathcal{S}_I$ be the semi-algebra consisting of all such cylinder
sets. Define an additive function
$
m_I:\mathcal{S}_I\to[0,1]
$
by
$
m_I([a_i]_{i\in\Lambda})=\prod_{i\in\Lambda}p_{a_i}.
$
The product $\sigma$-algebra is
$
\mathcal{B}_I=\sigma(\mathcal{S}_I).
$
By the standard extension theorem, the above set function extends uniquely
to a probability measure on $(\Omega_I,\mathcal{B}_I)$, still denoted by
$m_I$.

If $I=\mathbb{N}_0$, define
$
\sigma^+:\Omega_{\mathbb{N}_0}\to\Omega_{\mathbb{N}_0}
$
by
$
\sigma^+((y_n)_{n=0}^{+\infty})=(x_n)_{n=0}^{+\infty},
~
x_n=y_{n+1},~ n\geq 0.
$
The probability preserving system
$
(\Omega_{\mathbb{N}_0},\mathcal{B}_{\mathbb{N}_0},m_{\mathbb{N}_0},\sigma^+)
$
is called a one-sided Bernoulli shift.

If $I=\mathbb{Z}$, define
$
\sigma:\Omega_{\mathbb{Z}}\to\Omega_{\mathbb{Z}}
$
by
$
\sigma((y_n)_{n=-\infty}^{+\infty})=(x_n)_{n=-\infty}^{+\infty},
~
x_n=y_{n+1},~ n\in\mathbb{Z}.
$
The probability preserving system
$
(\Omega_{\mathbb{Z}},\mathcal{B}_{\mathbb{Z}},m_{\mathbb{Z}},\sigma)
$
is called a two-sided Bernoulli shift.

Two probability preserving systems
$
(X_i,\mathcal{B}_i,\mu_i,f_i),~ i=1,2,
$
where the maps $f_i$ are not necessarily invertible, are called
isomorphic if there exist $M_i\in\mathcal{B}_i$ satisfying the following
properties:
\begin{enumerate}
	\item[(i)] $\mu_i(M_i)=1$ and $f_i(M_i)\subset M_i$;
	\item[(ii)] there is a bimeasurable bijection $\pi:M_1\to M_2$ such that
	$\pi_*\mu_1=\mu_2$ and
	$
	\pi\circ f_1=f_2\circ\pi
	$
	on $M_1$.
\end{enumerate}

For a probability preserving endomorphism
$
(X,\mathcal{B},\mu,f),
$
we use the standard natural extension. More precisely, it is the
invertible probability preserving system
$
(\widetilde X,\widetilde{\mathcal B},\widetilde\mu,\widetilde f)
$
defined, up to measure theoretic isomorphism, as the inverse limit of
$
(X,\mathcal{B},\mu,f).
$
It can be realized as follows. Let
$
\widetilde X
=
\{(x_n)_{n\in\mathbb Z}\in X^{\mathbb Z}: f(x_n)=x_{n+1},\ n\in\mathbb Z\}.
$
Let
$
\widetilde f((x_n)_{n\in\mathbb Z})=(x_{n+1})_{n\in\mathbb Z}.
$
The $\sigma$-algebra $\widetilde{\mathcal B}$ is the restriction of the
product $\sigma$-algebra to $\widetilde X$, and $\widetilde\mu$ is the unique
$\widetilde f$-invariant probability measure such that
$
(\pi_n)_*\widetilde\mu=\mu
$
for every $n\in\mathbb Z$, where
$
\pi_n((x_j)_{j\in\mathbb Z})=x_n.
$
Then
$
\pi_0\circ\widetilde f=f\circ\pi_0.
$
See, for example, \cite[Theorem 1.7.1]{PU2010}.

\begin{definition}\label{def-Bernoulli}
	Let $(X,\mathcal{B},\mu,f)$ be a probability preserving system.
	We say that $\mu$ is a Bernoulli measure if the natural extension
	of $(X,\mathcal{B},\mu,f)$ is isomorphic to a two-sided Bernoulli shift.
	
	In particular, if $(X,\mathcal{B},\mu,f)$ itself is isomorphic to a one-sided
	Bernoulli shift, then $\mu$ is a Bernoulli measure, because the
	natural extension of a one-sided Bernoulli shift is the corresponding
	two-sided Bernoulli shift. If $f$ is invertible modulo $\mu$, this definition
	reduces to the usual requirement that $(X,\mathcal{B},\mu,f)$ be isomorphic
	to a two-sided Bernoulli shift.
	
	A probability measure $\mu$ is called a periodic measure if there is a
	periodic point $p\in X$ of period $n$ such that
	$
	\mu=\frac{1}{n}\sum_{j=0}^{n-1}\delta_{f^j p},
	$
	where $\delta_x$ is the Dirac measure at $x\in X$.
\end{definition}

\subsection{Basic classes of systems}\label{subsec-basic-classes}

We recall the three classes of systems considered in the main results. These classes will be denoted by $(S1)$, $(S2)$ and $(S3)$.

First, let $A=(A_{ij})_{0\leq i,j\leq k-1}$ be a zero-one matrix. 
The one-sided subshift of finite type defined by $A$ is the system 
$(\Sigma_A^+,\sigma_A)$, where
$$
\Sigma_A^+
=
\left\{
x=(x_n)_{n\geq 0}\in \{0,1,\ldots,k-1\}^{\mathbb{N}_0}:
A_{x_nx_{n+1}}=1,\ n\geq 0
\right\},
$$
and $\sigma_A:\Sigma_A^+\to \Sigma_A^+$ is the left shift defined by
$
(\sigma_A x)_n=x_{n+1},~ n\geq 0.
$
The two-sided subshift of finite type defined by $A$ is the system 
$(\Sigma_A,\sigma_A)$, where
$$
\Sigma_A
=
\left\{
x=(x_n)_{n\in\mathbb{Z}}\in \{0,1,\ldots,k-1\}^{\mathbb{Z}}:
A_{x_nx_{n+1}}=1,\ n\in\mathbb{Z}
\right\},
$$
and $\sigma_A:\Sigma_A\to \Sigma_A$ is the left shift defined by
$
(\sigma_A x)_n=x_{n+1},~ n\in\mathbb{Z}.
$
In this paper, a subshift of finite type may be either one-sided or two-sided. This is the class denoted by
\begin{enumerate}
	\item[(S1)] $f:X\to X$ is a subshift of finite type.
\end{enumerate}

Second, let $M$ be a compact Riemannian manifold and let $f:M\to M$ be a $C^{1+\varepsilon}$ diffeomorphism. A compact $f$-invariant set $X\subset M$ is called a hyperbolic set if there is a continuous $Df$-invariant splitting
$
T_XM=E^s\oplus E^u
$
and there exist constants $C>0$ and $0<\lambda<1$ such that, for every $x\in X$ and every $n\geq 0$,
$
\|Df^n_x v\|\leq C\lambda^n\|v\|,
~ v\in E_x^s,
$
and
$
\|Df^{-n}_x v\|\leq C\lambda^n\|v\|,
~ v\in E_x^u.
$
The hyperbolic set $X$ is called locally maximal if there exists an open neighborhood $U$ of $X$ such that
$
X=\bigcap_{n\in\mathbb{Z}}f^n(U).
$
This is the class denoted by
\begin{enumerate}
	\item[(S2)] $X$ is a locally maximal hyperbolic set of a $C^{1+\varepsilon}$ diffeomorphism $f$.
\end{enumerate}

Third, let $X$ be a compact Riemannian manifold and let $f:X\to X$ be a $C^{1+\varepsilon}$ map. We say that $X$ is a repeller  of a $C^{1+\varepsilon}$ map $f$ if $f^{-1}(X)=X$ and there exist constants $C>0$ and $\kappa>1$ such that
$
\|Df^n_x v\|\geq C\kappa^n\|v\|
$
for every $x\in X$, every $v\in T_xX$ and every $n\geq 1$. This is the class denoted by
\begin{enumerate}
	\item[(S3)] $X$ is a repeller  of a $C^{1+\varepsilon}$ map $f$.
\end{enumerate}

Finally, we recall the topological assumptions used below. A dynamical system $(X,f)$ is called topologically transitive if, for any two nonempty open sets $U,V\subset X$, there exists $n\geq 0$ such that
$
f^n(U)\cap V\neq\emptyset.
$
It is called topologically mixing if, for any two nonempty open sets $U,V\subset X$, there exists $N\geq 1$ such that
$
f^n(U)\cap V\neq\emptyset
$
for every $n\geq N$. In the cases $(S2)$ and $(S3)$, these topological properties are always understood for the restricted system $f|_X:X\to X$.

\subsection{Thermodynamic formalism and equilibrium measure}
We recall the notion of topological pressure and topological entropy. Since we need to deal with non-compact sets, we use the notion   introduced by Pesin and Pitskel as a Carathéodory characteristic. Suppose that  $(X,f)$ is  a dynamical system.  
Let $\mathcal{U}$ be a finite open cover of $X$. We denote by $\mathcal{W}_n(\mathcal{U})$ the collection of words $\mathbf{U}=\left(U_0, \ldots, U_n\right) \in \mathcal{U}^{n+1}$. For each $\mathbf{U} \in \mathcal{W}_n(\mathcal{U})$ we write $m(\mathbf{U})=n$, and define the open set
$$
X(\mathbf{U})=\left\{x \in X: f^k x \in U_k \text { for } k=0, \ldots, n\right\} .
$$
We say that the collection of words $\Gamma \subset \bigcup_{n \geq 1} \mathcal{W}_n(\mathcal{U})$ covers the set $Z \subset X$ provided that $\bigcup_{\mathrm{U} \in \mathrm{\Gamma}} X(\mathrm{U}) \supset Z$.
Consider a continuous function $\varphi: X \rightarrow \mathbb{R}$. Given $\mathbf{U} \in \mathcal{W}_n(\mathcal{U})$ we write
\begin{equation}\label{equ-CC}
	\varphi(\mathbf{U})= \begin{cases}\sup _{X(\mathbf{U})} \sum_{k=0}^{n-1} \varphi \circ f^k & \text { if } X(\mathbf{U}) \neq \emptyset, \\ -\infty & \text { if } X(\mathbf{U})=\emptyset .\end{cases}
\end{equation}
For each set $Z \subset X$ and each number $\alpha \in \mathbb{R}$ let

$$
M(Z, \alpha, \varphi, \mathcal{U})=\lim _{n \rightarrow \infty} \inf _{\Gamma} \sum_{\mathbf{U} \in \Gamma} \exp (-\alpha m(\mathbf{U})+\varphi(\mathbf{U})),
$$
where the infimum is taken over all collections of words $\Gamma \subset \bigcup_{k \geq n} \mathcal{W}_k(\mathcal{U})$ covering $Z$. Let diam $\mathcal{U}$ be the diameter of the cover $\mathcal{U}$. 
\begin{definition}
	The topological pressure of $\varphi$ on the set $Z$ (with respect to $f$ ) is defined by:
	$$
	P_Z(\varphi) \stackrel{\text { def }}{=}  \lim _{\operatorname{diam} \mathcal{U} \rightarrow 0} \inf \{\alpha: M(Z, \alpha, \varphi, \mathcal{U})=0\}.
	$$
	We call $h_{top}(f, Z)\stackrel{\text { def }}{=} P_Z(0)$ the Bowen topological entropy of $f$ on $Z$.
\end{definition}
We emphasize that $Z$ need not be compact nor $f$-invariant.
When $Z=X$ (and thus $Z$ is compact and $f$-invariant), the number $P_X(\varphi)$ coincides with the notion of topological pressure for compact sets introduced by Ruelle in the case of expansive maps, and by Walters in the general case \cite{Walters1982}. We refer to \cite{Barreira1996,Pesin1997} for references and further details.

\begin{definition}
	Given $\mu\in\mathcal{M}_f(X).$ Let  $\xi=\{A_1,  \cdots,  A_n\}$ be a finite partition of $X$,   define
	$H_\mu(\xi)=-\sum_{i=1}^n\mu(A_i)\log\mu(A_i).$  We denote by $\bigvee_{i=0}^{n-1}f^{-i}\xi$ the partition whose element is the set $\bigcap_{i=0}^{n-1}f^{-i}A_{j_i},  1\leq j_i\leq n$.   Then the limit
	$h_\mu(f,  \xi)=\lim\limits_{n\to\infty}\frac1n H_\mu\left(\bigvee_{i=0}^{n-1}f^{-i}\xi\right)$ exists
	and we define the metric entropy of $\mu$ as
	$$h_{\mu}(f):=\sup\left\{h_\mu(f,  \xi):\xi~\textrm{is a finite measurable partition of X}\right\}. $$
\end{definition}

According to the variational principle from \cite{Walters1975}, $$P_X(\varphi)=\sup\left\{h_\mu(f)+ \int \varphi d \mu:\mu\in \mathcal{M}_f(X)\right\}.$$
\begin{definition}
	A measure $\mu\in \mathcal{M}_f(X)$ is said to be an {equilibrium measure} associated with $\varphi$ if $P_X(\varphi)=h_\mu(f)+ \int \varphi d \mu.$
\end{definition}

\begin{lemma}\cite[Lemma 2.4]{Qiu2011}\label{lemma-qiu}
	Assume that the entropy map $\mu\mapsto h_\mu(f)$ is upper semicontinuous on $\mathcal{M}_f(X)$ and that $\varphi\in C(X)$. If $\varphi$ has a unique equilibrium measure $\mu_\varphi$, then 
	$$
	\int \psi d \mu_\varphi=\lim_{t\to 0}\frac{P_X(\varphi+t\psi)-P_X(\varphi)}{t}, \quad\forall \psi \in C(X).
	$$
\end{lemma}

\subsection{Hausdorff dimension and $v$-dimension} 

Let $X$ be a separable metric space. Given a set $Z \subset X$ and $\alpha>0$, we define
$
m(Z, \alpha)=\lim\limits_{\delta \rightarrow 0} \inf _{\mathcal{U}} \sum_{U \in \mathcal{U}}(\operatorname{diam} U)^\alpha
$
where the infimum is taken over all finite or countable cover $\mathcal{U}$ of $Z$ by sets of diameter at most $\delta$. There exists a unique value of $\alpha$ at which $m(Z, \alpha)$ jumps from $+\infty$ to $0$. This value is called Hausdorff dimension of $Z$ and is denoted by $\operatorname{dim}_H Z$. We have
$$
\operatorname{dim}_H Z=\inf \{\alpha: m(Z, \alpha)=0\} .
$$
For every Borel probability measure $\mu$ on $X$, we define the Hausdorff dimension of $\mu$ by
$$
\operatorname{dim}_H \mu=\inf \left\{\operatorname{dim}_H Z: \mu(Z)=1\right\} .
$$

We also recall a Carathéodory dimension characteristic introduced by Barreira and Schmeling in \cite{BS2000}. 
Suppose that  $(X,f)$ is  a dynamical system.  Let $\mathcal{U}$ be a finite open cover of $X$. Let now $v: X \rightarrow \mathbb{R}$ be a strictly positive continuous function. For each $\mathbf{U} \in \mathcal{W}_n(\mathcal{U})$ we define $v(\mathbf{U})$ as in (\ref{equ-CC}).
For each set $Z \subset X$ and each number $\alpha \in \mathbb{R}$, we define
$
N(Z, \alpha, v, \mathcal{U})=\lim _{n \rightarrow \infty} \inf _{\Gamma} \sum_{\mathbf{U} \in \Gamma} \exp (-\alpha v(\mathbf{U})),
$
where the infimum is taken over all collections of words $\Gamma \subset \bigcup_{k \geq n} \mathcal{W}_k(\mathcal{U})$ covering $Z$. Set
$
\operatorname{dim}_{v, \mathcal{U}} Z=\inf \{\alpha: N(Z, \alpha, v, \mathcal{U})=0\} .
$
One can easily show that the limit
$$
\operatorname{dim}_v Z \stackrel{\text { def }}{=} \lim _{\operatorname{diam} \mathcal{U} \rightarrow 0} \operatorname{dim}_{v, \mathcal{U}} Z
$$
exists, and we call it the $v$-dimension of $Z$ (with respect to $f$ ).
For every Borel probability measure $\mu$ on $X$, let
$
\operatorname{dim}_{v, \mathcal{U}} \mu=\inf \left\{\operatorname{dim}_{v, \mathcal{U}} Z: \mu(Z)=1\right\} .
$
The limit
$$
\operatorname{dim}_v \mu \stackrel{\text { def }}{=} \lim _{\operatorname{diam} \mathcal{U} \rightarrow 0} \operatorname{dim}_{v, \mathcal{U}} \mu
$$
exists, and is called the $v$-dimension of $\mu$.
\begin{lemma}\label{lemma-F}\cite{BS2000}
	Suppose that  $(X,f)$ is  a dynamical system.  Then for any $\mu\in\mathcal{M}^e_f(X),$ we have $\dim_v\mu=\frac{h_\mu(f)}{\int v d\mu}$.
\end{lemma}
  
\subsection{Maps between Banach spaces}
Let $V, W$ be Banach vector spaces. Denote by $L_s^k(V, W)$ the space of symmetric $k$-linear maps from the $k$-fold product $V^k:=V \times \cdots \times V$ into $W$. Given $T_k \in L_s^k(V, W)$ and $(H, \cdots, H) \in V^k$ we write $T_k\left(H^k\right):=T_k(H, \cdots, H)$. 
\begin{definition}
	Let $U\subset V$ be open. A mapping $\Gamma:U\to W$ is said to be analytic at $u\in U$ if there exist $\varepsilon>0$ and maps
	$
	T_k=T_k(u)\in L_s^k(V,W)$ for any $k\ge 1,
	$
	such that $u+H\in U$ whenever $\|H\|_V<\varepsilon$, and
	$$
	\Gamma(u+H)=\Gamma(u)+\sum_{k=1}^{\infty}\frac{1}{k!}T_k(H^k)
	$$
	for every $H\in V$ with $\|H\|_V<\varepsilon$, where the series converges absolutely and uniformly on $B_V(0,\varepsilon)$. We say that $\Gamma$ is analytic in $U$ if it is analytic at every point $u\in U$.
\end{definition}

\begin{definition}
	Let $U\subset V$ be open and $u\in U$. We say that a mapping $\Gamma:U\to W$ is Fr\'echet differentiable at $u$ if there exists a bounded linear operator $A\in L(V,W)$ such that
	$$
	\lim_{\|H\|_V\to 0}\frac{\|\Gamma(u+H)-\Gamma(u)-A(H)\|_W}{\|H\|_V}=0.
	$$
	Such $A$ is unique, it is denoted by $D\Gamma(u)$ and called the Fr\'echet derivative of $\Gamma$ at $u$.
\end{definition}

\begin{definition}
	Let $U\subset V$ be open, $u\in U$, and $H\in V$. We say that $\Gamma:U\to W$ is G\^ateaux differentiable at $u$ in the direction $H$ if the limit
	$$
	d\Gamma(u;H):=\lim_{t\to 0}\frac{\Gamma(u+tH)-\Gamma(u)}{t}=\left.\frac{d}{d t} \Gamma(u+t H)\right|_{t=0}
	$$
	exists in $W$.  $d\Gamma(u;H)$ is called the G\^ateaux derivative of $\Gamma$ at $u$ in the direction $H$.
\end{definition}

We have the following properties.
\begin{proposition}\label{prop-frech}
	Let $\Gamma: U \subset V \rightarrow W.$
	\begin{enumerate}
		\item[(1)] If  $\Gamma$  is Fréchet differentiable at $u \in U,$ then for any $H,$  $\Gamma$  is G\^ateaux differentiable at $u \in U$ in the direction $H,$ and $d \Gamma(u; H)=D\Gamma(u)(H).$
		\item[(2)]  If $\Gamma$ is analytic on $U$, then $\Gamma$ is Fr\'echet differentiable in $U$ and the map
		$
		D\Gamma:U\to L(V,W)
		$
		is analytic.
	\end{enumerate}
\end{proposition}

\section{Analytic paths for prescribed averages and entropy}\label{sec-hyper-entropy}

In this section we prove Theorem \ref{Thm-int}. In fact, we establish a slightly stronger statement in terms of the $v$-dimension
from which the entropy version follows by taking $v\equiv 1$. Our strategy is to construct analytic families of equilibrium states whose quotient Birkhoff averages move in a controlled way, while the associated $v$-dimension varies analytically and realizes the prescribed endpoint values.

\subsection{Statement in the $v$-dimensional form}

To formulate the result, let $d\in\mathbb N$. Given
$$
\Phi=(\varphi_1,\ldots,\varphi_d)\in C(X)^d,
\qquad
\Psi=(\psi_1,\ldots,\psi_d)\in C(X)^d,
$$
we assume throughout that $\psi_i>0$ for every $1\le i\le d$.
Given $\alpha=\left(\alpha_1, \ldots, \alpha_d\right) \in \mathbb{R}^d$ we set
$$
K_\alpha=K_\alpha(\Phi, \Psi;f)=\bigcap_{i=1}^d\left\{x \in X: \lim _{n \rightarrow \infty} \frac{\varphi_{i, n}(x)}{\psi_{i, n}(x)}=\alpha_i\right\},
$$
where
$
\varphi_{i, n}(x)=\sum_{k=0}^{n-1} \varphi_i\left(f^k x\right)$ and $\psi_{i, n}(x)=\sum_{k=0}^{n-1} \psi_i\left(f^k x\right) .
$
Define a continuous function $\mathcal{P}=\mathcal{P}^{(\Phi, \Psi)}: \mathcal{M}_f(X) \rightarrow \mathbb{R}^d$ by
$$
\mathcal{P}(\mu)=\left(\frac{\int_X \varphi_1 \mathrm{~d} \mu}{\int_X \psi_1 \mathrm{~d} \mu}, \ldots, \frac{\int_X \varphi_d \mathrm{~d} \mu}{\int_X \psi_d \mathrm{~d} \mu}\right) .
$$
For any $1\leq i \leq d,$ denote $\mathcal{P}_i(\mu)=\frac{\int_X \varphi_i\mathrm{~d} \mu}{\int_X \psi_i \mathrm{~d} \mu}.$
Since $\mathcal{M}_f(X)$ is compact and connected, and $\mathcal{P}$ is continuous, the set $\mathcal{P}(\mathcal{M}_f(X))$ is also compact and connected.

We formulate the stronger $v$-dimensional version, the entropy statement is obtained by taking $v\equiv1$.
\begin{theoremalph}\label{Thm-int-2}
	Let $(X,f)$ be a dynamical system satisfying one of the conditions $(S1)$, $(S2)$ and $(S3)$.
	Assume that $(X,f)$ is topologically mixing
	(resp. topologically transitive). Given $d\in\mathbb N^+$, $0<\gamma\leq 1$,
	$v\in C^\gamma(X)$ with $v>0$, and
	$
	(\Phi,\Psi)\in C^\gamma(X)^d\times C^\gamma(X)^d,
	$
	$
	\psi_i>0,~ 1\leq i\leq d.
	$
	Then for any
	$
	\hat\alpha,\tilde\alpha
	\in
	\operatorname{Int}\mathcal P(\mathcal M_f(X)),
	$
	and any
	$
	0<\hat h\leq\dim_v K_{\hat\alpha},
	$ $
	0<\tilde h\leq\dim_v K_{\tilde\alpha},
	$
	there is a path
	$$
	\{\mu_t\}_{t\in[0,1]}\subset\mathcal M_f^B(X)
	\qquad
	\left(\hbox{resp. }
	\{\mu_t\}_{t\in[0,1]}\subset\mathcal M_f^e(X)\right)
	$$
	such that
	\begin{enumerate}
		\item $t\mapsto \int\psi\,d\mu_t$ is analytic for every H\"older continuous function $\psi$;
		
		\item $t\mapsto \dim_v\mu_t$ is analytic;
		
		\item
		$
		(\mathcal P(\mu_0),\dim_v\mu_0)
		=
		(\tilde\alpha,\tilde h),
		~
		(\mathcal P(\mu_1),\dim_v\mu_1)
		=
		(\hat\alpha,\hat h);
		$
		
		\item $t\mapsto\mathcal P(\mu_t)$ is analytic. Moreover, if
		$\hat\alpha=\tilde\alpha=\alpha$, then the path can be chosen so that
		$
		\mathcal P(\mu_t)=\alpha
		$
		for every $t\in[0,1];$
		
		\item $S_{\mu_t}=X$ for every $t\in[0,1]$.
	\end{enumerate}
	Moreover, in the special case
	$
	\psi_i\equiv1,
	~
	1\leq i\leq d,
	$
	the path can be chosen so that
	$
	\inf_{t\in[0,1]}\dim_v\mu_t
	=
	\min\{\hat h,\tilde h\},
	$
	and
	$
	\mathcal{P}_i(\mu_t)=t\hat{\alpha}_i+(1-t)\tilde{\alpha}_i
	$
	for every $1\leq i\leq d$ and every $t\in[0,1]$.
\end{theoremalph}

\begin{remark}
	The formulation in terms of $v$-dimension and quotient Birkhoff averages is motivated by applications to multifractal analysis. Quantities of the form
	$$
	\frac{\int \varphi\,d\mu}{\int \psi\,d\mu}
	\qquad\text{and}\qquad
	\lim_{n\to\infty}\frac{\sum_{k=0}^{n-1}\varphi(f^k x)}{\sum_{k=0}^{n-1}\psi(f^k x)}
	$$
	arise naturally in mixed multifractal spectra and conditional variational principles; see, for example, \cite{BS2001,IJ2015}. They also appear in higher-dimensional multifractal analysis, where several observables must be treated simultaneously; see \cite{BSS2002,JT2021}. This is why we work throughout with vector valued potentials and the $v$-dimensional formalism.
\end{remark}

\subsection{Auxiliary lemmas}
\subsubsection{Convex realization of quotient averages}
For any $r\in \mathbb{R},$ denote $r^+=\{s\in\mathbb{R}:s>r\}$ and $r^-=\{s\in\mathbb{R}:s<r\}.$  For any $d \in \mathbb{N},$ $r=\left(r_1, \ldots, r_{d}\right)\in\mathbb{R}^d$ and $\xi=\left(\xi_1, \ldots, \xi_{d}\right)\in\{+,-\}^d$, we define $$r^\xi=\{s=\left(s_1, \ldots, s_{d}\right)\in\mathbb{R}^d:s_i\in r_i^{\xi_i} \text{ for }i=1,2,\cdots,d\}.$$ We denote $F^d=\{\left(\frac{p_1}{q_1}, \ldots, \frac{p_d}{q_d}\right):p_i,q_i\in\mathbb{R}\text{ and }q_i>0 \text{ for any }1\leq i\leq d\}.$
It is easy to check that
\begin{lemma}\label{lemma-E}
	Let $b_i=\frac{p^i}{q^i}\in F^1$ for $i=1,2.$
	\begin{description}
		\item[(1)] If $b_1=b_2,$ then $\frac{\theta p^1+(1-\theta)p^2}{\theta q^1+(1-\theta)q^2}=b_1=b_2$ for any $\theta\in[0,1].$
		\item[(2)] If $b_1\neq b_2,$ then $\frac{\theta p^1+(1-\theta)p^2}{\theta q^1+(1-\theta)q^2}$ is strictly monotonic on $\theta\in[0,1].$
	\end{description}
\end{lemma}
We shall use the following elementary selection lemma.

\begin{lemma}\label{lemma-C}
	Let $d\in\mathbb N$ and
	$
	\alpha=\left(\frac{p_1}{q_1},\ldots,\frac{p_d}{q_d}\right)\in F^d,
	$
	where all denominators are positive. Suppose that
	$$
	b_\xi=
	\left(
	\frac{p_1^\xi}{q_1^\xi},\ldots,\frac{p_d^\xi}{q_d^\xi}
	\right)\in F^d,
	\qquad \xi\in\{+,-\}^d,
	$$
	are $2^d$ points such that $b_\xi\in\alpha^\xi$ for every
	$\xi\in\{+,-\}^d$. Then there exist numbers
	$\theta_\xi\in[0,1]$, $\xi\in\{+,-\}^d$, such that
	$
	\sum_{\xi\in\{+,-\}^d}\theta_\xi=1
	$
	and
	\begin{equation}\label{equ-BC}
		\frac{
			\sum_{\xi\in\{+,-\}^d}\theta_\xi p_i^\xi
		}{
			\sum_{\xi\in\{+,-\}^d}\theta_\xi q_i^\xi
		}
		=
		\frac{p_i}{q_i},
		\qquad 1\leq i\leq d.
	\end{equation}
\end{lemma}

\begin{proof}
	For $d=1$, the assertion is exactly
	Lemma \ref{lemma-E}. Assume that the assertion holds for some
	$d=k$. We prove it for $d=k+1$.
	Let
	$
	\alpha=
	\left(
	\frac{p_1}{q_1},\ldots,\frac{p_{k+1}}{q_{k+1}}
	\right)\in F^{k+1},
	$
	and let
	$
	b_\xi=
	\left(
	\frac{p_1^\xi}{q_1^\xi},\ldots,
	\frac{p_{k+1}^\xi}{q_{k+1}^\xi}
	\right),
	~ \xi\in\{+,-\}^{k+1},
	$
	satisfy $b_\xi\in\alpha^\xi$ for every $\xi$.
	First fix the subfamily with $\xi_{k+1}=+$. Applying the induction
	hypothesis to the first $k$ coordinates, we obtain numbers
	$\tau_\xi^+\in[0,1]$, $\xi_{k+1}=+$, such that
	$
	\sum_{\xi_{k+1}=+}\tau_\xi^+=1
	$
	and
	$
	\frac{
		\sum_{\xi_{k+1}=+}\tau_\xi^+ p_i^\xi
	}{
		\sum_{\xi_{k+1}=+}\tau_\xi^+ q_i^\xi
	}
	=
	\frac{p_i}{q_i},
	~ 1\leq i\leq k.
	$
	Set
	$
	P_i^+=\sum_{\xi_{k+1}=+}\tau_\xi^+p_i^\xi,
	~
	Q_i^+=\sum_{\xi_{k+1}=+}\tau_\xi^+q_i^\xi.
	$
	Then
	$
	\frac{P_i^+}{Q_i^+}=\frac{p_i}{q_i},
	~ 1\leq i\leq k.
	$
	Moreover, since $\xi_{k+1}=+$ implies
	$
	\frac{p_{k+1}^\xi}{q_{k+1}^\xi}>
	\frac{p_{k+1}}{q_{k+1}},
	$
	we also have
	$
	\frac{P_{k+1}^+}{Q_{k+1}^+}>
	\frac{p_{k+1}}{q_{k+1}}.
	$
	Similarly, applying the induction hypothesis to the subfamily with
	$\xi_{k+1}=-$, we obtain numbers
	$\tau_\xi^-\in[0,1]$, $\xi_{k+1}=-$, such that
	$
	\sum_{\xi_{k+1}=-}\tau_\xi^-=1.
	$
	Defining
	$
	P_i^-=\sum_{\xi_{k+1}=-}\tau_\xi^-p_i^\xi,
	~
	Q_i^-=\sum_{\xi_{k+1}=-}\tau_\xi^-q_i^\xi,
	$
	we have
	$
	\frac{P_i^-}{Q_i^-}=\frac{p_i}{q_i},
	~ 1\leq i\leq k,
	$
	and
	$
	\frac{P_{k+1}^-}{Q_{k+1}^-}<
	\frac{p_{k+1}}{q_{k+1}}.
	$
	By Lemma \ref{lemma-E}, there exists $\lambda\in(0,1)$ such that
	$
	\frac{
		\lambda P_{k+1}^+ +(1-\lambda)P_{k+1}^-
	}{
		\lambda Q_{k+1}^+ +(1-\lambda)Q_{k+1}^-
	}
	=
	\frac{p_{k+1}}{q_{k+1}}.
	$
	For $1\leq i\leq k$, since
	$
	P_i^+=\frac{p_i}{q_i}Q_i^+,
	~
	P_i^-=\frac{p_i}{q_i}Q_i^-,
	$
	we also get
	$
	\frac{
		\lambda P_i^+ +(1-\lambda)P_i^-
	}{
		\lambda Q_i^+ +(1-\lambda)Q_i^-
	}
	=
	\frac{p_i}{q_i}.
	$
	Finally define
	$$
	\theta_\xi=
	\begin{cases}
		\lambda\tau_\xi^+,
		&
		\xi_{k+1}=+,
		\\
		(1-\lambda)\tau_\xi^-,
		&
		\xi_{k+1}=-.
	\end{cases}
	$$
	Then $\theta_\xi\in[0,1]$ and
	$
	\sum_{\xi\in\{+,-\}^{k+1}}\theta_\xi=1.
	$
	The preceding identities give (\ref{equ-BC}).
	This completes the induction and hence the proof.
\end{proof}
From Lemma \ref{lemma-C} we have
\begin{corollary}\label{corollary-A}
	Suppose that  $(X,f)$ is  a dynamical system.  Let $d \in \mathbb{N}$ and $(\Phi, \Psi) \in C(X)^d \times C(X)^d.$  Then for any $\alpha\in \operatorname{Int}\mathcal{P}(\mathcal{M}_f(X)),$ and $2^d$ invariant measures $\{\mu_\xi\}_{\xi\in\{+,-\}^d}$ with 
	$\mathcal{P}\left(\mu_\xi\right)\in \alpha^\xi \text{ for any }\xi\in \{+,-\}^d,$ there are $2^d$ numbers $\{\theta_\xi\}_{\xi\in\{+,-\}^d}\subseteq [0,1]$ such that $\sum_{\xi\in\{+,-\}^d}\theta_\xi=1$ and $\mathcal{P}\left(\sum_{\xi\in\{+,-\}^d}\theta_\xi\mu_\xi\right)= \alpha.$
\end{corollary}

\begin{lemma}\label{lemma-D}
	Suppose that  $(X,f)$ is  a dynamical system.  Let $d \in \mathbb{N}$ and $(\Phi, \Psi) \in C(X)^d \times C(X)^d.$ If $A$ is a dense subset of $\mathcal{M}_f(X),$ then for any $\alpha\in \operatorname{Int}\mathcal{P}(\mathcal{M}_f(X)),$ there are $2^d$ invariant measures $\{\mu_\xi\}_{\xi\in\{+,-\}^d}\subset A$ such that { for any }$\xi\in \{+,-\}^d,$ one has
	$\mathcal{P}\left(\mu_\xi\right)\in \alpha^\xi.$
\end{lemma}
\begin{proof}
	Since $\alpha\in \operatorname{Int}\mathcal{P}(\mathcal{M}_f(X)),$ for any $\xi\in \{+,-\}^d$ there is $\nu_\xi\in \mathcal{M}_f(X)$ such that $\mathcal{P}\left(\nu_\xi\right)\in \alpha^\xi.$  Since $A$ is dense in $\mathcal{M}_f(X),$  for any $\xi\in \{+,-\}^d$  we can choose  $\mu_\xi\in  A$ close to $\nu_\xi$  such that 
	$\mathcal{P}\left(\mu_\xi\right)\in \alpha^\xi$.
\end{proof}

\begin{lemma}\label{lem:coordinate-correction}
	Let
	$
	\Phi=(\varphi_1,\ldots,\varphi_d)\in C(X)^d, $
	$
	\Psi=(\psi_1,\ldots,\psi_d)\in C(X)^d,
	$
	with $\psi_i>0$ for all $1\leq i\leq d$, and let
	$
	\Phi^{new}=(\varphi_1,\ldots,\varphi_d,\varphi_{d+1}),$
	$
	\Psi^{new}=(\psi_1,\ldots,\psi_d,\psi_{d+1}),
	$
	with $\varphi_{d+1},\psi_{d+1}\in C(X)$ and $\psi_{d+1}>0$.
	Denote
	$$
	\mathcal P=\mathcal P^{(\Phi,\Psi)}:\mathcal M_f(X)\to \mathbb R^d,
	\qquad
	\mathcal P^{new}=\mathcal P^{(\Phi^{new},\Psi^{new})}:\mathcal M_f(X)\to \mathbb R^{d+1}.
	$$
	Assume that $\alpha\in \operatorname{Int}\mathcal P(\mathcal M_f(X))$ and that
	$
	(\alpha,\alpha_{d+1}^1),~
	(\alpha,\alpha_{d+1}^2)
	\in
	\mathcal P^{new}(\mathcal M_f(X))
	$
	for some $\alpha_{d+1}^1<\alpha_{d+1}^2$.
	Then for every
	$
	\alpha_{d+1}\in(\alpha_{d+1}^1,\alpha_{d+1}^2)
	$
	and every $1\leq j\leq d$, there exist invariant measures
	$
	\eta_j^+,~\eta_j^-\in \mathcal M_f(X)
	$
	and positive numbers
	$
	\varepsilon_j^+,~\varepsilon_j^->0
	$
	such that
	$$
	\mathcal P^{new}(\eta_j^+)
	=
	(\alpha+\varepsilon_j^+e_j,\alpha_{d+1}),
	\qquad
	\mathcal P^{new}(\eta_j^-)
	=
	(\alpha-\varepsilon_j^-e_j,\alpha_{d+1}),
	$$
	where $e_j$ denotes the $j$-th standard basis vector of $~\mathbb R^d$.
\end{lemma}

\begin{proof}
	Choose $\mu^1,~\mu^2\in\mathcal M_f(X)$ such that
	$$
	\mathcal P^{new}(\mu^1)=(\alpha,\alpha_{d+1}^1),
	\qquad
	\mathcal P^{new}(\mu^2)=(\alpha,\alpha_{d+1}^2).
	$$
	For $s\in[0,1]$, set
	$
	\mu_s=(1-s)\mu^1+s\mu^2.
	$
	By Lemma \ref{lemma-E}, for every $1\leq i\leq d$ and every $s\in[0,1]$,
	$
	\mathcal P_i(\mu_s)=\alpha_i.
	$
	Moreover,
	$
	s\mapsto \mathcal P^{new}_{d+1}(\mu_s)
	$
	is continuous and strictly monotone from $\alpha_{d+1}^1$ to $\alpha_{d+1}^2$. Hence there exists $s_0\in(0,1)$ such that
	$
	\mathcal P^{new}(\mu_{s_0})=(\alpha,\alpha_{d+1}).
	$
	Since $\alpha\in\operatorname{Int}\mathcal P(\mathcal M_f(X))$, there exists $\varepsilon>0$ such that
	$
	\alpha+[-\varepsilon,\varepsilon]^d
	\subset
	\mathcal P(\mathcal M_f(X)).
	$
	Fix $1\leq j\leq d$. We may choose invariant measures
	$
	\nu_j^+,~\nu_j^-\in\mathcal M_f(X)
	$
	such that
	$$
	\mathcal P(\nu_j^+)=\alpha+\varepsilon e_j,
	\qquad
	\mathcal P(\nu_j^-)=\alpha-\varepsilon e_j.
	$$
	Now we construct $\eta_j^+$. For $\lambda\in[0,1]$ and $s\in[0,1]$, define
	$$
	M_{\lambda,s}^+
	=
	(1-\lambda)\mu_s+\lambda\nu_j^+.
	$$
	For every $i\neq j$, Lemma \ref{lemma-E} gives
	$
	\mathcal P_i(M_{\lambda,s}^+)=\alpha_i.
	$
	For $\lambda>0$, again by Lemma \ref{lemma-E},
	$
	\mathcal P_j(M_{\lambda,s}^+)>\alpha_j.
	$
	Consider the continuous function
	$
	(\lambda,s)\mapsto \mathcal P^{new}_{d+1}(M_{\lambda,s}^+).
	$
	At $\lambda=0$, its values at $s=0$ and $s=1$ are respectively
	$
	\alpha_{d+1}^1
	~\hbox{and}~
	\alpha_{d+1}^2.
	$
	Since
	$
	\alpha_{d+1}^1<\alpha_{d+1}<\alpha_{d+1}^2,
	$
	there exists $\lambda_+>0$ sufficiently small such that
	$$
	\mathcal P^{new}_{d+1}(M_{\lambda_+,0}^+)<\alpha_{d+1}
	<
	\mathcal P^{new}_{d+1}(M_{\lambda_+,1}^+).
	$$
	Then there exists $s_+\in(0,1)$ such that
	$
	\mathcal P^{new}_{d+1}(M_{\lambda_+,s_+}^+)=\alpha_{d+1}.
	$
	Set
	$
	\eta_j^+=M_{\lambda_+,s_+}^+.
	$
	Then
	$\mathcal P_i(\eta_j^+)=\alpha_i$
	for all $i\neq j,$
	while
	$
	\mathcal P_j(\eta_j^+)>\alpha_j.
	$
	Thus there exists $\varepsilon_j^+>0$ such that
	$
	\mathcal P^{new}(\eta_j^+)
	=
	(\alpha+\varepsilon_j^+e_j,\alpha_{d+1}).
	$
	Similarly, we obtain $\eta_j^-\in\mathcal M_f(X)$ and $\varepsilon_j^->0$ such that
	$
	\mathcal P^{new}(\eta_j^-)
	=
	(\alpha-\varepsilon_j^-e_j,\alpha_{d+1}).
	$
	The proof is complete.
\end{proof}

\begin{lemma}\label{lem:interior-lifting}
	Keep the notation of Lemma \ref{lem:coordinate-correction}. All interiors are taken with respect to the usual Euclidean topology of the corresponding ambient spaces. Assume that
	$
	\alpha\in \operatorname{Int}\mathcal P(\mathcal M_f(X))
	$
	and that
	$
	(\alpha,\alpha_{d+1}^1),~
	(\alpha,\alpha_{d+1}^2)
	\in
	\mathcal P^{new}(\mathcal M_f(X))
	$
	for some $\alpha_{d+1}^1<\alpha_{d+1}^2$. Then, for every
	$
	\alpha_{d+1}\in(\alpha_{d+1}^1,\alpha_{d+1}^2),
	$
	one has
	$
	(\alpha,\alpha_{d+1})
	\in
	\operatorname{Int}\mathcal P^{new}(\mathcal M_f(X)).
	$
\end{lemma}

\begin{proof}
	Write
	$
	C^{new}:=\mathcal P^{new}(\mathcal M_f(X)).
	$
	As in the proof of Lemma \ref{lem:coordinate-correction}, choose
	$\mu^1,\mu^2\in\mathcal M_f(X)$ and $s_0\in(0,1)$ such that
	$$
	\mathcal P^{new}(\mu^1)=(\alpha,\alpha_{d+1}^1),
	\qquad
	\mathcal P^{new}(\mu^2)=(\alpha,\alpha_{d+1}^2),
	$$
	and
	$$
	\mathcal P^{new}(\mu_*)=(\alpha,\alpha_{d+1}),
	\qquad
	\mu_*:=(1-s_0)\mu^1+s_0\mu^2.
	$$
	By Lemma \ref{lem:coordinate-correction}, for each $1\leq j\leq d$ there exist
	$\eta_j^+,\eta_j^-\in\mathcal M_f(X)$ and $\varepsilon_j^+,\varepsilon_j^->0$
	such that
	$$
	\mathcal P^{new}(\eta_j^+)
	=
	(\alpha+\varepsilon_j^+e_j,\alpha_{d+1}),
	\qquad
	\mathcal P^{new}(\eta_j^-)
	=
	(\alpha-\varepsilon_j^-e_j,\alpha_{d+1}).
	$$
	
	For 
	$
	\sigma=(\sigma_1,\ldots,\sigma_d,\sigma_{d+1})\in\{+,-\}^{d+1},
	$
	define
	$$
	\eta_j^\sigma=
	\begin{cases}
		\eta_j^+,& \sigma_j=+,\\
		\eta_j^-,& \sigma_j=-,
	\end{cases}
	\qquad 1\leq j\leq d,
	$$
	and
	$$
	\eta_{d+1}^\sigma=
	\begin{cases}
		\mu^2,& \sigma_{d+1}=+,\\
		\mu^1,& \sigma_{d+1}=-.
	\end{cases}
	$$
	Denote
	$$
	\varepsilon_j^\sigma=
	\begin{cases}
		\varepsilon_j^+,& \sigma_j=+,\\
		\varepsilon_j^-,& \sigma_j=-,
	\end{cases}
	\qquad 1\leq j\leq d.
	$$

		For $u=(u_1,\ldots,u_{d+1})$ close to $0$, define the affine combination
	$$
	m_\sigma(u)
	=
	\left(1-\sum_{k=1}^{d+1}u_k\right)\mu_*
	+
	\sum_{k=1}^{d+1}u_k\eta_k^\sigma.
	$$
	Here $m_\sigma(u)$ is only regarded as a finite signed invariant measure
	when some coordinates of $u$ are negative. Thus, instead of applying
	$\mathcal P^{new}$ directly to $m_\sigma(u)$, we define a map
	$F_\sigma$ near $0\in\mathbb R^{d+1}$ by the quotient formulas
	$$
	(F_\sigma)_i(u)
	=
	\frac{\int \varphi_i\,dm_\sigma(u)}
	{\int \psi_i\,dm_\sigma(u)},
	\qquad
	1\leq i\leq d+1.
	$$
	Since $m_\sigma(0)=\mu_*$ and all functions $\psi_i$ are strictly positive,
	all denominators are positive at $u=0$. Hence, after restricting $u$ to a
	sufficiently small Euclidean neighborhood of $0$, all denominators remain
	positive. Therefore $F_\sigma$ is real analytic in a neighborhood of $0$.
	Moreover, whenever $m_\sigma(u)\in\mathcal M_f(X)$, one has
	$
	F_\sigma(u)=\mathcal P^{new}(m_\sigma(u)).
	$
	
		We record the coordinate formulas for $F_\sigma$ at $u$. 	First fix $1\leq i\leq d$. Directly from the definition of $F_\sigma$,
		we have
		$$
		(F_\sigma)_i(u)-\alpha_i
		=
		\frac{
			\int\varphi_i\,dm_\sigma(u)
			-
			\alpha_i\int\psi_i\,dm_\sigma(u)
		}{
			\int\psi_i\,dm_\sigma(u)
		}.
		$$
		Using the definition of $m_\sigma(u)$, the numerator is equal to
		$$
		\left(1-\sum_{k=1}^{d+1}u_k\right)
		\left(
		\int\varphi_i\,d\mu_*
		-
		\alpha_i\int\psi_i\,d\mu_*
		\right)
		+
		\sum_{k=1}^{d+1}u_k
		\left(
		\int\varphi_i\,d\eta_k^\sigma
		-
		\alpha_i\int\psi_i\,d\eta_k^\sigma
		\right).
		$$
		By the choice of $\mu_*$ and of the measures $\eta_k^\sigma$, all terms
		in this expression vanish except possibly the term with $k=i$. Hence
		$$
		(F_\sigma)_i(u)-\alpha_i
		=
		\begin{cases}
			\dfrac{u_i\varepsilon_i^+\int\psi_i\,d\eta_i^+}
			{\int\psi_i\,dm_\sigma(u)},
			& \sigma_i=+,\\[2.2ex]
			-\dfrac{u_i\varepsilon_i^-\int\psi_i\,d\eta_i^-}
			{\int\psi_i\,dm_\sigma(u)},
			& \sigma_i=-.
		\end{cases}
		$$
		Similarly, from the definition of
	$F_\sigma$,
	$$
	(F_\sigma)_{d+1}(u)-\alpha_{d+1}
	=
	\frac{
		\int\varphi_{d+1}\,dm_\sigma(u)
		-
		\alpha_{d+1}\int\psi_{d+1}\,dm_\sigma(u)
	}{
		\int\psi_{d+1}\,dm_\sigma(u)
	}.
	$$
	By the choice of $\mu_*$ and of the measures $\eta_k^\sigma$, the numerator
	has only one possible nonzero contribution, namely the term with $k=d+1$.
	Since $\eta_{d+1}^\sigma=\mu^2$ if $\sigma_{d+1}=+$ and
	$\eta_{d+1}^\sigma=\mu^1$ if $\sigma_{d+1}=-$, we get
	$$
	(F_\sigma)_{d+1}(u)-\alpha_{d+1}
	=
	\begin{cases}
		\dfrac{u_{d+1}(\alpha_{d+1}^2-\alpha_{d+1})
			\int\psi_{d+1}\,d\mu^2}
		{\int\psi_{d+1}\,dm_\sigma(u)},
		& \sigma_{d+1}=+,\\[2.2ex]
		-\dfrac{u_{d+1}(\alpha_{d+1}-\alpha_{d+1}^1)
			\int\psi_{d+1}\,d\mu^1}
		{\int\psi_{d+1}\,dm_\sigma(u)},
		& \sigma_{d+1}=-.
	\end{cases}
	$$
	
		It follows from the preceding coordinate formulas that
	$$
	\frac{\partial (F_\sigma)_i}{\partial u_k}(0)=0
	\qquad
	\hbox{whenever }i\neq k.
	$$
		Indeed, each coordinate formula has the form
	$
	(F_\sigma)_i(u)-\alpha_i=\frac{c_i u_i}{D_i(u)},
	$
	where $c_i\neq0$ and $D_i(0)>0$. If $k\neq i$, then along the line
	$u=t e_k$ one has $u_i=0$, and hence
	$
	(F_\sigma)_i(t e_k)-\alpha_i=0.
	$
	Therefore
	$
	\frac{\partial (F_\sigma)_i}{\partial u_k}(0)=0.
	$
	Thus $DF_\sigma(0)$ is a diagonal matrix. Moreover, its diagonal entries
	are given by nonzero multiples of
	$$
	\varepsilon_i^\sigma,
	\qquad
	1\leq i\leq d,
	$$
	and of
	$$
	\alpha_{d+1}^2-\alpha_{d+1}
	\quad\hbox{or}\quad
	\alpha_{d+1}-\alpha_{d+1}^1
	$$
	for the last coordinate. Since
	$
	\varepsilon_i^\sigma>0,
	$ $
	\alpha_{d+1}^1<\alpha_{d+1}<\alpha_{d+1}^2,
	$
	and since every integral of a positive function $\psi_i$ with respect to an
	invariant probability measure is positive, all diagonal entries are nonzero.
	Therefore $DF_\sigma(0)$ is invertible.
	
	By the inverse function theorem, there exist open neighborhoods
	$
	U_\sigma\ni0,
	$ $
	V_\sigma\ni(\alpha,\alpha_{d+1})
	$
	such that
	$
	F_\sigma:U_\sigma\to V_\sigma
	$
	is a diffeomorphism. Shrinking $U_\sigma$ if necessary, we may also assume
	that
	$$
	\sum_{k=1}^{d+1}|u_k|<1
	\qquad
	\hbox{for all }u\in U_\sigma.
	$$
	
	Let $\mathcal O_\sigma$ be the open orthant in $\mathbb R^{d+1}$ determined by $\sigma$, and let $\overline{\mathcal O_\sigma}$ be its closure. We claim that
	$$
	V_\sigma\cap
	\bigl((\alpha,\alpha_{d+1})+\overline{\mathcal O_\sigma}\bigr)
	\subset
	C^{new}.
	$$
	Indeed, take
	$
	z\in
	V_\sigma\cap
	\bigl((\alpha,\alpha_{d+1})+\overline{\mathcal O_\sigma}\bigr).
	$
	Since $F_\sigma:U_\sigma\to V_\sigma$ is a diffeomorphism, there exists a unique $u\in U_\sigma$ such that
	$
	F_\sigma(u)=z.
	$
	The coordinate formulas determine the sign of each $u_k$. Indeed, for
	$1\leq k\leq d$, the quantity
	$(F_\sigma)_k(u)-\alpha_k=z_k-\alpha_k$ has the same sign as $u_k$ when
	$\sigma_k=+$, and has the opposite sign to $u_k$ when $\sigma_k=-$,
	because all denominators and all coefficients appearing in the formulas
	are positive. On the other hand, since
	$
	z\in(\alpha,\alpha_{d+1})+\overline{\mathcal O_\sigma},
	$
	we have $z_k-\alpha_k\geq0$ when $\sigma_k=+$ and
	$z_k-\alpha_k\leq0$ when $\sigma_k=-$. Hence $u_k\geq0$ for every
	$1\leq k\leq d$.
	The same reasoning applies to the last coordinate, using
	$\alpha_{d+1}^2-\alpha_{d+1}>0$ when $\sigma_{d+1}=+$ and
	$\alpha_{d+1}-\alpha_{d+1}^1>0$ when $\sigma_{d+1}=-$. Therefore
	$
	u_k\geq0,
	$ $
	1\leq k\leq d+1.
	$
	Since $u\in U_\sigma$ and $\sum_{k=1}^{d+1}|u_k|<1$, we have
	$
	\sum_{k=1}^{d+1}u_k<1.
	$
	Hence $m_\sigma(u)$ is a convex combination of invariant probability measures, and therefore
	$
	m_\sigma(u)\in\mathcal M_f(X).
	$
	By the definition of $F_\sigma$, we have
	$
	z=F_\sigma(u)=\mathcal P^{new}(m_\sigma(u))\in C^{new}.
	$
	This proves the claim.
	
	For each $\sigma\in\{+,-\}^{d+1}$, choose $r_\sigma>0$ such that
	$
	B((\alpha,\alpha_{d+1}),r_\sigma)\subset V_\sigma.
	$
	Let
	$$
	r=\min_{\sigma\in\{+,-\}^{d+1}}r_\sigma>0.
	$$
	Then, for every $\sigma\in\{+,-\}^{d+1}$,
	$
	B((\alpha,\alpha_{d+1}),r)
	\cap
	\bigl((\alpha,\alpha_{d+1})+\overline{\mathcal O_\sigma}\bigr)
	\subset
	C^{new}.
	$
	Since
	$$
	B((\alpha,\alpha_{d+1}),r)
	=
	\bigcup_{\sigma\in\{+,-\}^{d+1}}
	\left[
	B((\alpha,\alpha_{d+1}),r)
	\cap
	\bigl((\alpha,\alpha_{d+1})+\overline{\mathcal O_\sigma}\bigr)
	\right],
	$$
	one has
	$
	B((\alpha,\alpha_{d+1}),r)\subset C^{new}.
	$
	Therefore
	$
	(\alpha,\alpha_{d+1})\in\operatorname{Int}C^{new}.
	$
	The proof is complete.
\end{proof}

\subsubsection{Analytic dependence of equilibrium states}
\begin{proposition}\label{proposition-standard-properties}
	Let $(X,f)$ be a topologically transitive dynamical system satisfying one of the conditions $(S1)$, $(S2)$ and $(S3)$. Then the following statements hold.
	\begin{enumerate}
		\item For any $0<\gamma\leq 1$, the pressure function
		$
		P_X:C^\gamma(X)\to \mathbb{R}
		$
		is analytic; see \cite{Ruelle1977}.
		
		\item Every H\"older continuous potential $\varphi$ has a unique equilibrium measure $\mu_\varphi$. Moreover, $\mu_\varphi$ is an ergodic measure with full support. In particular,
		$
		S_{\mu_\varphi}=X.
		$
		See \cite{Bowen2008,Ruelle1977,PU2010}.
		
		\item The set of periodic measures is weak$^*$ dense in $\mathcal{M}_f(X)$. This follows from the specification property, or from the spectral decomposition together with specification on the mixing components; see \cite{DenkerGrillenbergerSigmund1976,PU2010}.
		
		\item The entropy map
		$
		\mathcal{M}_f(X)\ni \mu\mapsto h_\mu(f)
		$
		is upper semi-continuous with respect to the weak$^*$ topology; see \cite{Ruelle1977,PU2010}.
		
		\item If, in addition, $(X,f)$ is topologically mixing, then for every H\"older continuous potential $\varphi$, the equilibrium measure $\mu_\varphi$ is Bernoulli in the sense of Definition \ref{def-Bernoulli}; see \cite{Bowen2008,Ruelle1977,PU2010}.
	\end{enumerate}
\end{proposition}

\begin{lemma}\label{prop-analytic}
	Let $(X,f)$ be a topologically transitive dynamical system satisfying one of the conditions $(S1)$, $(S2)$ and $(S3)$.  For each H\"older continuous function $\varphi$, denote by $\mu_\varphi$ the unique equilibrium measure of $\varphi$. Let $\psi$ be a H\"older continuous function. Then  for any  $0<\gamma\leq 1,$ the function $\varphi\mapsto \int \psi d\mu_\varphi$ is analytic in $C^\gamma(X).$ 
\end{lemma}
\begin{proof}
		Fix a H\"older continuous function $\psi$. Then there exists $\eta>0$ such that
		$\psi\in C^\eta(X)$. Let
		$
		\beta:=\min\{\gamma,\eta\}.
		$
		Then
		$
		C^\gamma(X)\subset C^\beta(X)
		$
		and 
		$
		\psi\in C^\beta(X).
		$
		By Proposition \ref{proposition-standard-properties}, the pressure map
		$
		P_X:C^\beta(X)\to \mathbb R
		$
		is analytic. Hence by Proposition \ref{prop-frech} its Fr\'echet derivative
		$$
		DP_X:C^\beta(X)\to \mathcal L(C^\beta(X),\mathbb R)
		$$
		is also analytic. Since $\psi\in C^\beta(X)$ is fixed, the map
		$$
		T\mapsto T(\psi), \qquad T\in \mathcal L(C^\beta(X),\mathbb R),
		$$
		is a continuous linear functional on $\mathcal L(C^\beta(X),\mathbb R)$.
		Therefore the map
		$
		\phi\mapsto DP_X(\phi)(\psi)
		$
		is analytic on $C^\beta(X)$.
		For every $\phi\in C^\gamma(X)\subset C^\beta(X)$, by Proposition \ref{prop-frech} one has
		$$
		DP_X(\phi)(\psi)
		=
		\left.\frac{d}{dt}P_X(\phi+t\psi)\right|_{t=0}.
		$$
		By Lemma \ref{lemma-qiu} and Proposition \ref{proposition-standard-properties}, we have
		$$
		\left.\frac{d}{dt}P_X(\phi+t\psi)\right|_{t=0}
		=
		\int_X \psi\, d\mu_\phi.
		$$
		Since $C^\gamma(X)$ is continuously embedded into $C^\beta(X)$, the
		restriction of an analytic map on $C^\beta(X)$ to $C^\gamma(X)$ is
		still analytic. Therefore
		$
		\phi\mapsto \int_X \psi\, d\mu_\phi
		$
		is analytic on $C^\gamma(X)$. 
\end{proof}
Given vectors $\alpha=\left(\alpha_1, \ldots, \alpha_d\right) \in \mathbb{R}^d$ and $\Phi=\left(\varphi_1, \ldots, \varphi_d\right) \in C(X)^d$ we shall write
$$
\alpha * \Phi=\left(\alpha_1 \varphi_1, \ldots, \alpha_d \varphi_d\right) \in C(X)^d \quad \text { and } \quad\langle\alpha, \Phi\rangle=\sum_{i=1}^d \alpha_i \varphi_i \in C(X) .
$$
\begin{lemma}[{\cite[Theorems 8, 12 and 13]{BSS2002}}]\label{lemma-analytic}
	Let $(X,f)$ be a topologically transitive dynamical system satisfying one
	of the conditions $(S1)$, $(S2)$ and $(S3)$. Given $d\in\mathbb N^+$,
	$0<\gamma\leq1$, $v\in C^\gamma(X)$ with $v>0$, and
	$
	(\Phi,\Psi)\in C^\gamma(X)^d\times C^\gamma(X)^d
	$
	with
	$
	\psi_i>0
	$
	for every $1\leq i\leq d$. Then, for every $\alpha\in\mathbb R^d$, the following statements hold.
	\begin{enumerate}
		\item[(1)] If
		$
		\alpha\notin \mathcal{P}(\mathcal{M}_f(X)),
		$
		then
		$
		K_\alpha=\emptyset.
		$
		
		\item[(2)] If
		$
		\alpha\in \operatorname{Int}\mathcal{P}(\mathcal{M}_f(X)),
		$
		then
		$
		K_\alpha\neq\emptyset
		$
		and
		$$
		\dim_v K_\alpha
		=
		\max
		\left\{
		\frac{h_\mu(f)}{\int v\,d\mu}:
		\mu\in\mathcal{M}_f(X),\ 
		\mathcal{P}(\mu)=\alpha
		\right\}.
		$$
		Moreover, there exist $q(\alpha)\in\mathbb R^d$ and an ergodic equilibrium measure
		$\mu_\alpha$ of the H\"older continuous potential
		$
		\left\langle q(\alpha),\Phi-\alpha*\Psi\right\rangle
		-
		(\dim_v K_\alpha)v
		$
		such that
		$$
		\mathcal{P}(\mu_\alpha)=\alpha,
		\qquad
		\mu_\alpha(K_\alpha)=1,
		\qquad
		\dim_v\mu_\alpha=\dim_vK_\alpha.
		$$
		
		\item[(3)] If further $(X,f)$ is topologically mixing, then the maps
		$
		\alpha\mapsto \dim_v K_\alpha
		$
		and $\alpha\mapsto q(\alpha)$
		are  analytic in
		$
		\operatorname{Int}\mathcal{P}(\mathcal{M}_f(X)).\footnote{The analytic dependence of $\alpha\mapsto q(\alpha)$ is given in Page 87 of \cite{BSS2002}.}
		$
	\end{enumerate}
\end{lemma}

\begin{lemma}\label{lemma-continuous}
	Let $(X,f)$ be a dynamical system satisfying that the entropy map
	$
	\mathcal{M}_f(X)\ni \mu\mapsto h_\mu(f)
	$
	is upper semi-continuous. Given $d\in\mathbb N^+$,
	$0<\gamma\leq1$, $v\in C^\gamma(X)$ with $v>0$, and
	$
	(\Phi,\Psi)\in C^\gamma(X)^d\times C^\gamma(X)^d
	$
	with
	$
	\psi_i>0
	$
	for every $1\leq i\leq d$.  For any vectors $\alpha\in \mathcal{P}(\mathcal{M}_f(X)),$ denote $$\dim_v(\alpha):=\sup\left\{\frac{h_\mu(f)}{\int v d\mu}: \mu\in\mathcal{M}_f(X) ~\text{and }~\mathcal{P}(\mu)=\alpha\right\}.$$ 
	Then the function $\alpha\mapsto \dim_v(\alpha)$ is upper semicontinuous on $\mathcal{P}(\mathcal{M}_f(X)).$ In particular,   it is continuous at any $\beta$ with  $\dim_v(\beta)=0.$
\end{lemma}
\begin{proof}
	Fix
	$
	\alpha\in\mathcal P(\mathcal M_f(X))
	$
	and write
	$
	F(\alpha)
	=
	\{\mu\in\mathcal M_f(X):\mathcal P(\mu)=\alpha\}.
	$
	Since $\psi_i>0$ for every $1\leq i\leq d$, the map $\mathcal P$ is continuous.
	Hence $F(\alpha)$ is a nonempty compact subset of $\mathcal M_f(X)$.
	
	Since $v>0$, the map
	$
	\mu\mapsto\int v\,d\mu
	$
	is continuous and bounded away from zero on $\mathcal M_f(X)$. Therefore
	$
	\mu\mapsto h_\mu(f)/\int v\,d\mu
	$
	is upper semi-continuous. Hence the supremum in the definition of
	$
	\dim_v(\alpha)
	$
	is attained, that is,
	$$
	\dim_v(\alpha)
	=
	\max
	\left\{
	\frac{h_\mu(f)}{\int v\,d\mu}:
	\mu\in\mathcal M_f(X)
	\text{ and }
	\mathcal P(\mu)=\alpha
	\right\}.
	$$
	
	Let
	$
	\alpha_n\to\alpha
	$
	in
	$
	\mathcal P(\mathcal M_f(X)).
	$
	For each $n$, choose
	$
	\mu_n\in F(\alpha_n)
	$
	such that
	$
	\dim_v(\alpha_n)
	=
	\frac{h_{\mu_n}(f)}{\int v\,d\mu_n}.
	$
	To prove upper semi-continuity, it is enough to consider a subsequence along which
	$
	\dim_v(\alpha_n)
	$
	converges to
	$
	\limsup\limits_{n\to\infty}\dim_v(\alpha_n).
	$
	Passing to a further subsequence if necessary, compactness of $\mathcal M_f(X)$ gives
	$
	\mu_n\to\mu
	$
	for some
	$
	\mu\in\mathcal M_f(X).
	$
	By continuity of $\mathcal P$,
	$
	\mathcal P(\mu)=\alpha,
	$
	and hence
	$
	\mu\in F(\alpha).
	$
	Using the upper semi-continuity of the entropy map and the continuity of
	$
	\mu\mapsto\int v\,d\mu,
	$
	we obtain
	$$
	\limsup_{n\to\infty}\dim_v(\alpha_n)
	\leq
	\frac{h_\mu(f)}{\int v\,d\mu}
	\leq
	\dim_v(\alpha).
	$$
	Thus
	$
	\alpha\mapsto\dim_v(\alpha)
	$
	is upper semi-continuous on
	$
	\mathcal P(\mathcal M_f(X)).
	$
	
	If
	$
	\dim_v(\beta)=0,
	$
	then, since entropy is nonnegative,
	$
	\dim_v(\alpha)\geq0
	$
	for every
	$
	\alpha\in\mathcal P(\mathcal M_f(X)).
	$
	Therefore upper semi-continuity gives
	$$
	0
	\leq
	\liminf_{\alpha\to\beta}\dim_v(\alpha)
	\leq
	\limsup_{\alpha\to\beta}\dim_v(\alpha)
	\leq
	\dim_v(\beta)
	=
	0.
	$$
	Hence
	$
	\alpha\mapsto\dim_v(\alpha)
	$
	is continuous at $\beta$.
\end{proof}

\subsection{Proof of Theorem \ref{Thm-int-2}}

\begin{proof}[Proof of Theorem \ref{Thm-int-2}]
	(1) Suppose that $(X,f)$ is topologically mixing.  Fix $\hat{\alpha},\tilde{\alpha}\in \operatorname{Int}\mathcal{P}(\mathcal{M}_f(X))$, $0< \hat{h}\leq \dim_v K_{\hat{\alpha}}$ and $0<  \tilde{h}\leq  \dim_v K_{\tilde{\alpha}}.$  By Proposition \ref{proposition-standard-properties} and Lemma \ref{lemma-D}  there are periodic measures $\{\mu_\xi^i\}_{\xi\in\{+,-\}^d},$ $i=1,2,3,4$ such that 
	$$\mathcal{P}\left(\mu_\xi^1\right),\mathcal{P}\left(\mu_\xi^2\right)\in \hat{\alpha}^\xi, \mathcal{P}\left(\mu^3_\xi\right),\mathcal{P}\left(\mu^4_\xi\right)\in \tilde{\alpha}^\xi ~\text{for any }~\xi\in \{+,-\}^d,$$
	\begin{equation}\label{equ-AA}
		\{\mu_\xi^i\}_{\xi\in\{+,-\}^d}\cap\{\mu_\xi^j\}_{\xi\in\{+,-\}^d}=\emptyset ~\text{for any }~i\neq j\in\{1,2,3,4\}.
	\end{equation}
	We choose the periodic measures successively so that their average vectors are
	pairwise distinct. This is possible because the required sign conditions are
	open conditions and because periodic measures are dense in $\mathcal M_f(X)$.
	Thus the chosen periodic measures are pairwise distinct.

	By Corollary \ref{corollary-A} for any $i=1,2,3,4,$ there  are $\{\theta^i_\xi\}_{\xi\in\{+,-\}^d}\subseteq [0,1],$   such that $\sum_{\xi\in\{+,-\}^d}\theta^i_\xi=1$ and $$\mathcal{P}\left(\mu^1\right)=\mathcal{P}\left(\mu^2\right)= \hat{\alpha},\qquad \mathcal{P}\left(\mu^3\right)=\mathcal{P}\left(\mu^4\right)= \tilde{\alpha},$$
	where $\mu^i=\sum_{\xi\in\{+,-\}^d}\theta^i_\xi\mu^i_\xi.$
	Note that for any periodic measures $\omega_1$ and $\omega_2$, one has $\omega_1\neq \omega_2$ if and only if $S_{\omega_1}\cap S_{\omega_2}=\emptyset.$ From (\ref{equ-AA}), we have $S_{\mu^1}\cup S_{\mu^3}$ and $S_{\mu^2}\cup S_{\mu^4}$ are two disjoint closed sets. Define
	$$\varphi_{d+1}(x):=\frac{\rho(x,S_{\mu^1}\cup S_{\mu^3})-\rho(x,S_{\mu^2}\cup S_{\mu^4})}{\rho(x,S_{\mu^1}\cup S_{\mu^3})+\rho(x,S_{\mu^2}\cup S_{\mu^4})} ~\text{ for any }~x\in X.$$
	It is easy to check that $\varphi_{d+1}$ is a Lipschitz continuous function from $X$ to $[-1,1]$ and 
	$$\varphi_{d+1}^{-1}(-1)=S_{\mu^1}\cup S_{\mu^3},\qquad \varphi_{d+1}^{-1}(1)=S_{\mu^2}\cup S_{\mu^4}.$$
	
	We next show that $h_\mu(f)=0$ if 
	\begin{equation}\label{equ-A}
		\int \varphi_{d+1}d\mu=1,-1.
	\end{equation}
	Since
	$
	-1\leq \varphi_{d+1}\leq 1,
	$
	if $\mu\in\mathcal M_f(X)$ satisfies
	$
	\int \varphi_{d+1}\,d\mu=1,
	$
	then
	$
	\varphi_{d+1}=1
	$
	for $\mu$-almost every point. Hence $\mu$ is concentrated on
	$
	\varphi_{d+1}^{-1}(1)=S_{\mu^2}\cup S_{\mu^4}.
	$
	The set $S_{\mu^2}\cup S_{\mu^4}$ is a finite union of periodic orbits.
	Therefore every invariant probability measure concentrated on it has
	metric entropy zero. Thus
	$
	h_\mu(f)=0.
	$
	Similarly, if
	$
	\int \varphi_{d+1}\,d\mu=-1,
	$
	then $\mu$ is concentrated on
	$
	\varphi_{d+1}^{-1}(-1)=S_{\mu^1}\cup S_{\mu^3},
	$
	which is again a finite union of periodic orbits. Hence in this case also
	$
	h_\mu(f)=0.
	$
	
	Define $\psi_{d+1}\equiv 1,$ and let 
	$$
	\Phi^{new}=\left(\varphi_1, \ldots, \varphi_d,\varphi_{d+1}\right) \quad \text { and } \quad \Psi^{new}=\left(\psi_1, \ldots, \psi_d,\psi_{d+1}\right) .
	$$
	Denote $\mathcal{P}^{new}=\mathcal{P}^{(\Phi^{new}, \Psi^{new})}$.
	For any vectors $\beta=\left(\beta_1, \ldots, \beta_{d+1}\right) \in \mathbb{R}^{d+1}$ denote $$\dim_v(\beta):=\sup\left\{\frac{h_\mu(f)}{\int v d\mu}: \mu\in\mathcal{M}_f(X) ~\text{and }~\mathcal{P}^{new}(\mu)=\beta\right\}.$$ 
	For any vectors $\beta=\left(\beta_1, \ldots, \beta_{d+1}\right) \in \mathcal{P}^{new}(\mathcal{M}_f(X))$ and any $q\in \mathbb{R}^{d+1},$
	let $\mu_{\beta,q}$ be the unique equilibrium measure of the function $\langle q, \Phi^{new}-\beta*\Psi ^{new}\rangle-v\dim_v(\beta)$. 
	From Lemma \ref{lemma-analytic}, for any  $\beta\in \operatorname{Int}\mathcal{P}^{new}(\mathcal{M}_f(X)),$   there is $q(\beta)\in \mathbb{R}^{d+1}$ such that $\mu_{\beta,q(\beta)}$ satisfies 
	$$\mathcal{P}^{new}(\mu_{\beta,q(\beta)})=\beta, \qquad \frac{h_{\mu_{\beta,q(\beta)}}(f)}{\int v d\mu_{\beta,q(\beta)}}=\dim_v(\beta).$$
	Moreover, $q(\beta)$ and $\dim_v(\beta)$ are analytic in $\operatorname{Int}\mathcal{P}^{new}(\mathcal{M}_f(X)).$
	Thus $\langle q(\beta), \Phi^{new}-\beta*\Psi^{new} \rangle-v\dim_v(\beta)$ is analytic in $\operatorname{Int}\mathcal{P}^{new}(\mathcal{M}_f(X)).$
	Combining with  lemma \ref{prop-analytic}, for any H\"older continuous function $\psi,$ the function $$\beta\mapsto \int \psi d\mu_{\beta,q(\beta)}$$ is analytic in $\operatorname{Int}\mathcal{P}^{new}(\mathcal{M}_f(X)).$
	
	From (\ref{equ-A}) we have 
	\begin{equation}\label{equ-BB}
		\dim_v(\hat{\alpha},\pm 1)=\dim_v(\tilde{\alpha},\pm 1)=0.
	\end{equation} 
	Since $(\hat{\alpha},\pm 1),(\tilde{\alpha},\pm 1)\in \mathcal{P}^{new}(\mathcal{M}_f(X)),$ then by Lemma \ref{lem:interior-lifting}, for every $a\in(-1,1)$ the points
	$(\hat\alpha,a)$ and $(\tilde\alpha,a)$ belong to
	$\operatorname{Int}\mathcal P^{new}(\mathcal M_f(X))$. Hence, by Lemma
	\ref{lemma-analytic}, the maps
	$$
	a\mapsto \dim_v(\hat\alpha,a),
	\qquad
	a\mapsto \dim_v(\tilde\alpha,a)
	$$
	are analytic, and thus continuous, on $(-1,1)$.
	At the endpoints, by (\ref{equ-BB})
	and Lemma \ref{lemma-continuous}, the spectrum is continuous. Thus both maps are continuous on $[-1,1]$.
	Note that  
	$$\max_{\alpha_{d+1}\in [-1,1]}\dim_v(\hat{\alpha},\alpha_{d+1})=\max \left\{\frac{h_\mu(f)}{\int v d\mu}: \mu \in \mathcal{M}_f(X) \text { and } \mathcal{P}(\mu)=\hat{\alpha}\right\}=\dim_v K_{\hat{\alpha}}$$
	and
	$$\max_{\alpha_{d+1}\in [-1,1]}\dim_v(\tilde{\alpha},\alpha_{d+1})=\max \left\{\frac{h_\mu(f)}{\int v d\mu}: \mu \in \mathcal{M}_f(X) \text { and } \mathcal{P}(\mu)=\tilde{\alpha}\right\}=\dim_v K_{\tilde{\alpha}}.$$
	So there are $\hat{\alpha}_{d+1},\tilde{\alpha}_{d+1}\in (-1,1)$ such that 
	$\dim_v(\hat{\alpha},\hat{\alpha}_{d+1})=\hat{h}$ and $\dim_v(\tilde{\alpha},\tilde{\alpha}_{d+1})=\tilde{h}.$  It's clear that $(\hat{\alpha},\hat{\alpha}_{d+1}),(\tilde{\alpha},\tilde{\alpha}_{d+1})\in \operatorname{Int}\mathcal{P}^{new}(\mathcal{M}_f(X)).$
	Set
	$$
	\hat{\beta}:=(\hat{\alpha},\hat{\alpha}_{d+1}),
	\qquad
	\tilde{\beta}:=(\tilde{\alpha},\tilde{\alpha}_{d+1}).
	$$
	Then
	$
	\hat{\beta},\tilde{\beta}
	\in
	\operatorname{Int}\mathcal P^{new}(\mathcal M_f(X)),
	$
	and
	$
	\dim_v(\hat{\beta})=\hat h,
	$
	$
	\dim_v(\tilde{\beta})=\tilde h.
	$
	For $\beta\in\operatorname{Int}\mathcal P^{new}(\mathcal M_f(X))$, put
	$$
	G_\beta
	=
	\left\langle q(\beta),
	\Phi^{new}-\beta*\Psi^{new}
	\right\rangle
	-
	v\dim_v(\beta).
	$$
	Let
	$
	\hat{\nu}:=\mu_{\hat{\beta},q(\hat{\beta})},
	$
	$
	\tilde{\nu}:=\mu_{\tilde{\beta},q(\tilde{\beta})}.
	$
	Then
	$
	\mathcal P^{new}(\hat{\nu})=\hat{\beta},
	$
	$
	\mathcal P^{new}(\tilde{\nu})=\tilde{\beta},
	$
	and
	$
	\dim_v\hat{\nu}=\hat h,
	$
	$
	\dim_v\tilde{\nu}=\tilde h.
	$
	
	Define a H\"older potential
	$$
	G_t=(1-t)G_{\tilde{\beta}}+tG_{\hat{\beta}},
	\qquad
	0\leq t\leq1,
	$$
	and let $\mu_t$ be the unique equilibrium measure of $G_t$.
	Since $G_t$ is a real analytic curve in $C^\gamma(X)$, Lemma \ref{prop-analytic} implies that, for
	every H\"older continuous function $\psi$,
	$$
	t\mapsto \int\psi\,d\mu_t
	$$
	is analytic. This gives property (1). In particular,
	$
	t\mapsto \int v\,d\mu_t
	$
	is analytic and strictly positive.
	
	Moreover, by Proposition \ref{proposition-standard-properties},
	$
	t\mapsto P_X(G_t)
	$
	is analytic. Since
	$$
	h_{\mu_t}(f)
	=
	P_X(G_t)-\int G_t\,d\mu_t,
	$$
	and since
	$$
	t\mapsto \int G_t\,d\mu_t
	=
	(1-t)\int G_{\tilde{\beta}}\,d\mu_t
	+
	t\int G_{\hat{\beta}}\,d\mu_t
	$$
	is analytic, it follows that
	$
	t\mapsto h_{\mu_t}(f)
	$
	is analytic. Therefore
	$$
	t\mapsto \dim_v\mu_t
	=
	\frac{h_{\mu_t}(f)}{\int v\,d\mu_t}
	$$
	is analytic. This gives property (2).
	
	Since $\mu_0$ is the unique equilibrium measure of $G_{\tilde{\beta}}$ and
	$\mu_1$ is the unique equilibrium measure of $G_{\hat{\beta}}$, we have
	$
	\mu_0=\tilde{\nu},
	$
	$
	\mu_1=\hat{\nu}.
	$
	Consequently,
	$$
	(\mathcal P(\mu_0),\dim_v\mu_0)
	=
	(\tilde{\alpha},\tilde h),
	\qquad
	(\mathcal P(\mu_1),\dim_v\mu_1)
	=
	(\hat{\alpha},\hat h).
	$$
	This gives property (3).
	
	For each $1\leq i\leq d$, both functions
	$
	t\mapsto \int\varphi_i\,d\mu_t
	$
	and 
	$
	t\mapsto \int\psi_i\,d\mu_t
	$
	are analytic, and the latter is strictly positive. Hence
	$
	t\mapsto
	\mathcal P_i(\mu_t)
	=
	\frac{\int\varphi_i\,d\mu_t}{\int\psi_i\,d\mu_t}
	$
	is analytic. Thus $t\mapsto\mathcal P(\mu_t)$ is analytic. This gives property (4).
	
	Finally, by Proposition \ref{proposition-standard-properties}, every $\mu_t$ is a Bernoulli
	measure and satisfies
	$
	S_{\mu_t}=X.
	$ This gives property (5).
	
	If, in addition, $\hat{\alpha}=\tilde{\alpha}=\alpha$, then we choose the
	path in a different way. Define
	$$
	\beta(t)
	=
	\left(
	\alpha,
	(1-t)\tilde{\alpha}_{d+1}+t\hat{\alpha}_{d+1}
	\right),
	\qquad
	0\leq t\leq1.
	$$
	If $\tilde{\alpha}_{d+1}=\hat{\alpha}_{d+1}$, then $\beta(t)$ is constant.
	Otherwise, Lemma \ref{lem:interior-lifting} gives
	$$
	\beta(t)\in
	\operatorname{Int}\mathcal P^{new}(\mathcal M_f(X))
	\qquad
	\hbox{for every }0<t<1.
	$$
	The endpoints also belong to
	$\operatorname{Int}\mathcal P^{new}(\mathcal M_f(X))$ by construction.
	Now define
	$$
	\mu_t:=\mu_{\beta(t),q(\beta(t))},
	\qquad
	0\leq t\leq1.
	$$
	The preceding analytic-dependence argument gives properties (1), (2), (3)
	and (5). Moreover,
	$
	\mathcal P^{new}(\mu_t)=\beta(t),
	$
	and hence
	$
	\mathcal P(\mu_t)=\alpha
	$
	for every 
	$t\in[0,1].$
	This proves the additional assertion in the case
	$\hat{\alpha}=\tilde{\alpha}$.

	It remains to prove the final assertion in the special case
	$
	\psi_i\equiv1,
	~
	1\leq i\leq d.
	$
	In this case $\psi_{d+1}\equiv1$ as well, and therefore
	$\mathcal P^{new}$ is affine. Hence
	$
	C^{new}:=\mathcal P^{new}(\mathcal M_f(X))
	$
	is a compact convex subset of $\mathbb R^{d+1}$. Since
	$\tilde{\beta}$ and $\hat{\beta}$ are interior points of $C^{new}$, the
	segment
	$
	\beta(t)=(1-t)\tilde{\beta}+t\hat{\beta},
	~
	0\leq t\leq1,
	$
	is contained in $\operatorname{Int}C^{new}$. Define again
	$$
	\mu_t:=\mu_{\beta(t),q(\beta(t))}.
	$$
	Then the same argument as above gives properties (1)--(5). Moreover,
	$
	\mathcal P^{new}(\mu_t)=\beta(t),
	$
	and hence, for every $1\leq i\leq d$,
	$$
	\mathcal P_i(\mu_t)
	=
	\beta_i(t)
	=
	t\hat{\alpha}_i+(1-t)\tilde{\alpha}_i.
	$$
	
	We next prove the lower bound for $\dim_v\mu_t$. Since
	$\mathcal P^{new}$ is affine and
	$
	\mathcal P^{new}(\tilde{\nu})=\tilde{\beta},
	$
	$
	\mathcal P^{new}(\hat{\nu})=\hat{\beta},
	$
	we have
	$
	\mathcal P^{new}((1-t)\tilde{\nu}+t\hat{\nu})
	=
	\beta(t).
	$
	Therefore, by the definition of $\dim_v(\beta(t))$,
	$$
	\dim_v(\beta(t))
	\geq
	\frac{
		h_{(1-t)\tilde{\nu}+t\hat{\nu}}(f)
	}{
		\int v\,d((1-t)\tilde{\nu}+t\hat{\nu})
	}
	=
	\frac{
		(1-t)h_{\tilde{\nu}}(f)+t h_{\hat{\nu}}(f)
	}{
		(1-t)\int v\,d\tilde{\nu}+t\int v\,d\hat{\nu}
	}.
	$$
	By Lemma \ref{lemma-E}, this quotient lies between
	$
	\dim_v\tilde{\nu}=\tilde h
	$
	and
	$
	\dim_v\hat{\nu}=\hat h.
	$
	Thus
	$$
	\dim_v\mu_t
	=
	\dim_v(\beta(t))
	\geq
	\min\{\hat h,\tilde h\}
	\qquad
	\hbox{for every }t\in[0,1].
	$$
	Since
	$
	\dim_v\mu_0=\tilde h,
	$
	$
	\dim_v\mu_1=\hat h,
	$
	we obtain
	$
	\inf_{t\in[0,1]}\dim_v\mu_t
	=
	\min\{\hat h,\tilde h\}.
	$
	This proves the special assertion when $\psi_i\equiv1$ for all
	$1\leq i\leq d$.

	(2) Now assume $(X,f)$ is topologically transitive. Fix $\hat{\alpha},\tilde{\alpha}\in \operatorname{Int}\mathcal{P}(\mathcal{M}_f(X))$, $0< \hat{h}\leq \dim_v K_{\hat{\alpha}}$ and $0<  \tilde{h}\leq  \dim_v K_{\tilde{\alpha}}.$   If $(X,f)$ is mixing, the result follows directly from the mixing case. Thus we assume that $(X,f)$ is transitive but not mixing.
	It is well known that $(X,f)$ admits a cyclic decomposition: there exist an
	integer $p\geq2$ and pairwise disjoint compact sets
	$$
	X=X_0\sqcup X_1\sqcup\cdots\sqcup X_{p-1}
	$$
	such that
	$$
	f(X_j)=X_{j+1\pmod p},
	\qquad
	0\leq j\leq p-1,
	$$
	and
	$$
	g:=f^p|_{X_0}:X_0\to X_0
	$$
	is topologically mixing. Moreover, the induced system $(X_0,g)$ belongs to
	the same class as $(X,f)$: it is a mixing subshift of finite type in case
	$(S1)$, a mixing locally maximal hyperbolic set for the diffeomorphism
	$f^p$ in case $(S2)$, and a mixing repeller for the map $f^p$ in
	case $(S3)$. See, for example,
	\cite[Proposition 4.5.6]{LM1995},
	\cite[Theorem 18.3.1]{KH1995}, and
	\cite[Theorem 11.2.15]{VianaOliveira2016}.

			For each $1\le i\le d$, define
		   $$
			\varphi_i^{(p)}=\sum_{k=0}^{p-1}\varphi_i\circ f^k,
			\qquad
			\psi_i^{(p)}=\sum_{k=0}^{p-1}\psi_i\circ f^k
			$$
			on $X_0$, and write
			$$
			\Phi^{(p)}=(\varphi_1^{(p)},\ldots,\varphi_d^{(p)}),\qquad
			\Psi^{(p)}=(\psi_1^{(p)},\ldots,\psi_d^{(p)}).
			$$
			Since $\psi_i>0$, we also have $\psi_i^{(p)}>0$. Denote $\mathcal{P}^{(p)}=\mathcal{P}^{(\Phi^{(p)}, \Psi^{(p)})}$.
			Define 
			$
			v^{(p)}=\sum_{k=0}^{p-1}v\circ f^k
			$
			on $X_0$. Then $v^{(p)}\in C^\gamma(X_0)$ and $v^{(p)}>0.$
			
			Define
			$$
			L:\mathcal{M}_g(X_0)\to\mathcal{M}_f(X),\qquad
			L(\nu)=\frac1p\sum_{j=0}^{p-1}(f^j)_*\nu.
			$$
			Then $L(\nu)$ is $f$-invariant. Moreover
			$$
			\int_X \varphi_i\,dL(\nu)
			=
			\frac1p\sum_{j=0}^{p-1}\int_{X_0}\varphi_i\circ f^j\,d\nu
			=
			\frac1p\int_{X_0}\varphi_i^{(p)}\,d\nu,
			$$
			and similarly
			$
			\int_X \psi_i\,dL(\nu)
			=
			\frac1p\int_{X_0}\psi_i^{(p)}\,d\nu.
			$
			Hence
			$
			\mathcal P(L(\nu))=\mathcal P^{(p)}(\nu).
			$
			Now define
			$$
			R:\mathcal{M}_f(X)\to\mathcal{M}_g(X_0),\qquad
			R(\mu)=p\,\mu|_{X_0}.
			$$
			Then $R(\mu)\in\mathcal M_g(X_0)$.
			Now we prove that $R\circ L=\mathrm{id}_{\mathcal M_g(X_0)}$. Let
			$\nu\in\mathcal M_g(X_0)$ and let $B\subset X_0$ be a Borel set. Since
			$f^j(X_0)= X_j$ and $X_j\cap X_0=\emptyset$ for $1\leq j\leq p-1$, we have
			$(f^j)^{-1}B=\emptyset$ as a subset of $X_0$ for $1\leq j\leq p-1$. Therefore
			$$
			R(L(\nu))(B)
			=
			p\,L(\nu)(B)
			=
			p\cdot\frac1p\nu(B)
			=
			\nu(B).
			$$
			Thus $R\circ L=\mathrm{id}_{\mathcal M_g(X_0)}$.
			Next we prove that $L\circ R=\mathrm{id}_{\mathcal M_f(X)}$. Let
			$\mu\in\mathcal M_f(X)$ and let $A\subset X$ be a Borel set. Write
			$
			A_j=A\cap X_j,
			~
			0\leq j\leq p-1.
			$
			Since $f^j(X_0)= X_j$, we have
			$
			X_0\cap f^{-j}A=X_0\cap f^{-j}A_j.
			$
			Moreover, $f^{-j}A_j\subset X_0$. 
			Therefore
			\begin{equation*}
				\begin{split}
					L(R(\mu))(A)
					&=
					\frac1p\sum_{j=0}^{p-1}(f^j)_*(R(\mu))(A)=\frac1p\sum_{j=0}^{p-1}R(\mu)(f^{-j}A)\\
					&=
					\sum_{j=0}^{p-1}\mu(X_0\cap f^{-j}A)=\sum_{j=0}^{p-1}\mu(f^{-j}A_j)\\
					&=
					\sum_{j=0}^{p-1}\mu(A_j)=\mu(A).
				\end{split}
			\end{equation*}
			Thus $L\circ R=\mathrm{id}_{\mathcal M_f(X)}$.
			Consequently, $L$ is a bijection between $\mathcal M_g(X_0)$ and
			$\mathcal M_f(X)$, with inverse $R$. Since
			$
			\mathcal P(L(\nu))=\mathcal P^{(p)}(\nu)
			$
			for every $\nu\in\mathcal M_g(X_0)$, we obtain
			$
			\mathcal P^{(p)}(\mathcal M_g(X_0))
			=
			\mathcal P(\mathcal M_f(X)).
			$
			In particular,
			$
			\hat\alpha,\tilde\alpha\in
			\operatorname{Int}\mathcal P^{(p)}(\mathcal M_g(X_0)).
			$

			Let
			$
			\nu_j=(f^j)_*\nu,
			~
			0\leq j\leq p-1.
			$
			Then $\nu_0=\nu$, $\nu_j\in\mathcal M_{f^p}(X_j)$, and
			$
			L(\nu)=\frac1p\sum_{j=0}^{p-1}\nu_j.
			$
			We first compare the entropies of the systems induced by $f^p$ on the
			cyclic components. Fix $0\leq j\leq p-1$, where the index is taken modulo
			$p$, and set
			$$
			S_j=f|_{X_j}:X_j\to X_{j+1},
			\qquad
			T_j=f^{p-1}|_{X_{j+1}}:X_{j+1}\to X_j.
			$$
			Then
			$
			T_j\circ S_j=f^p|_{X_j},
			~
			S_j\circ T_j=f^p|_{X_{j+1}}.
			$
			Moreover,
			$
			(S_j)_*\nu_j=\nu_{j+1},
			~
			(T_j)_*\nu_{j+1}=\nu_j.
			$
			The first identity shows that $S_j$ is a factor map from
			$
			(X_j,f^p|_{X_j},\nu_j)
			$
			to
			$
			(X_{j+1},f^p|_{X_{j+1}},\nu_{j+1}),
			$
			and the second identity shows that $T_j$ is a factor map in the opposite
			direction. Hence, we have
			$
			h_{\nu_{j+1}}(f^p|_{X_{j+1}})
			\leq
			h_{\nu_j}(f^p|_{X_j})
			$
			and
			$
			h_{\nu_j}(f^p|_{X_j})
			\leq
			h_{\nu_{j+1}}(f^p|_{X_{j+1}}).
			$
			Therefore one has
			$
			h_{\nu_{j+1}}(f^p|_{X_{j+1}})
			=
			h_{\nu_j}(f^p|_{X_j}).
			$
			Iterating this equality gives
			$
			h_{\nu_j}(f^p|_{X_j})
			=
			h_\nu(f^p|_{X_0})
			$
			for every $0\leq j\leq p-1$.
			
			Since the sets $X_0,\ldots,X_{p-1}$ are pairwise disjoint and invariant under
			$f^p$,  we have
			\begin{equation*}
				\begin{split}
					h_{L(\nu)}(f^p)
					=
					h_{\frac1p\sum_{j=0}^{p-1}\nu_j}(f^p)
					=
					\frac1p\sum_{j=0}^{p-1}h_{\nu_j}(f^p|_{X_j})
					=
					h_\nu(f^p|_{X_0}).
				\end{split}
			\end{equation*}
			Consequently,
			$$
			h_\nu(g)
			=
			h_\nu(f^p|_{X_0})
			=
			h_{L(\nu)}(f^p)
			=
			p\,h_{L(\nu)}(f).
			$$
			Furthermore,
			$\int_X v dL(\nu)=\frac{1}{p}\int_{X_0} v^{(p)} d \nu.$
			Then for any $\nu\in \mathcal M_g(X_0)$ we have 
			\begin{equation}
				\frac{h_\nu(g)}{\int_{X_0} v^{(p)}d\nu}=\frac{h_{L(\nu)}(f)}{\int_X vdL(\nu)}.
			\end{equation}
			Combining with Lemma \ref{lemma-analytic}, for any $\alpha\in \operatorname{Int}\mathcal{P}(\mathcal{M}_f(X))$ we have 
			\begin{equation}
				\begin{split}
					\dim_v K_\alpha&=\max \left\{\frac{h_\mu(f)}{\int v d\mu}: \mu \in \mathcal{M}_f(X) \text { and } \mathcal{P}(\mu)=\alpha\right\}\\
					&=\max \left\{\frac{h_{R(\mu)}(g)}{\int v^{(p)} dR(\mu)}: \mu \in \mathcal{M}_f(X) \text { and } \mathcal{P}(\mu)=\alpha\right\}\\
					&=\max \left\{\frac{h_{\nu}(g)}{\int v^{(p)} d\nu}: \nu \in \mathcal{M}_g(X_0) \text { and } \mathcal{P}^{(p)}(\nu)=\alpha\right\}\\
					&=\dim_{v^{(p)}} K_\alpha^{(p)},
				\end{split}
			\end{equation}
			where $K_\alpha^{(p)}=K_\alpha(\Phi^{(p)},\Psi^{(p)};g).$
			Thus we have $0< \hat{h}\leq \dim_{v^{(p)}} K_{\hat{\alpha}}^{(p)},$ and $0<  \tilde{h}\leq  \dim_{v^{(p)} }K_{\tilde{\alpha}}^{(p)}.$

			Applying the mixing case of the theorem to the system $(X_0,g)$ with potentials $(\Phi^{(p)},\Psi^{(p)})$ and  $v^{(p)}$, we obtain a path
			$
			\{\nu_t\}_{t\in[0,1]}\subset \mathcal M_g^B(X_0)
			$
			satisfying the corresponding properties.
			Define
			$$
			\mu_t=L(\nu_t)=\frac1p\sum_{j=0}^{p-1}(f^j)_*\nu_t.
			$$
			Then $\mu_t\in\mathcal M_f^e(X)$. Indeed, let $A$ be a Borel set with $f^{-1}A=A$, then $A\cap X_0$ is $g$-invariant, because
			$$
			g^{-1}(A\cap X_0)
			=
			X_0\cap f^{-p}A
			=
			X_0\cap A.
			$$
			Hence
			$
			\nu_t(A\cap X_0)\in\{0,1\}
			$
			by ergodicity of $\nu_t$ under $g$. Moreover,
			$$
			\mu_t(A)
			=
			\frac1p\sum_{j=0}^{p-1}(f^j)_*\nu_t(A)
			=
			\frac1p\sum_{j=0}^{p-1}\nu_t(f^{-j}A)=
			\frac1p\sum_{j=0}^{p-1}\nu_t(A\cap X_0)
			=
			\nu_t(A\cap X_0)\in\{0,1\}.
			$$
			Therefore $\mu_t$ is ergodic. 
			
			Now let $\psi$ be a H\"older continuous function. For each $0\le j\le p-1$, the restriction
			$
			\psi\circ f^j|_{X_0}
			$
			is also a H\"older continuous function. Since the mixing case yields that
			$
			t\mapsto \int_{X_0}\phi\,d\nu_t
			$
			is analytic for every H\"older continuous function $\phi$, it follows that each map
			$
			t\mapsto \int_{X_0}\psi\circ f^j\,d\nu_t
			$
			is analytic. On the other hand,
			$$
			\int_X \psi\,d\mu_t
			=
			\frac1p\sum_{j=0}^{p-1}\int_X \psi\,d\left((f^j)_*\nu_t\right)
			=
			\frac1p\sum_{j=0}^{p-1}\int_{X_0}\psi\circ f^j\,d\nu_t.
			$$
			Therefore $t\mapsto \int_X\psi\,d\mu_t$ is analytic as a finite sum of analytic functions.
			
			Since
			$
			\dim_{v^{(p)}} \nu_t=\dim_v\mu_t,
			$
			the map
			$
			t\mapsto \dim_v\mu_t
			$
			is analytic. Since
			$
			\mathcal P(\mu_t)=\mathcal P^{(p)}(\nu_t),
			$
			the analyticity of $t\mapsto\mathcal P(\mu_t)$ follows from the corresponding
			property for $\nu_t$.
			Moreover, by property (3) for the path $\{\nu_t\}_{t\in[0,1]}$,
			$$
			(\mathcal P^{(p)}(\nu_0),\dim_{v^{(p)}} \nu_0)=(\tilde\alpha,\tilde h),
			\qquad
			(\mathcal P^{(p)}(\nu_1),\dim_{v^{(p)}}\nu_1)=(\hat\alpha, \hat h).
			$$
			Hence
			$$
			(\mathcal P(\mu_0),\dim_v\mu_0)=(\tilde\alpha,\tilde h),
			\qquad
			(\mathcal P(\mu_1),\dim_v\mu_1)=(\hat\alpha,\hat h).
			$$
			Finally, since $S_{\nu_t}=X_0$, we have
			$
			S_{(f^j)_*\nu_t}=X_j,
			$
			and therefore
			$
			S_{\mu_t}=\bigcup_{j=0}^{p-1}X_j=X.
			$
			
			If, moreover, $\psi_i\equiv1$ for all $1\leq i\leq d$, then we apply the
			special part of the mixing case to the normalized potentials
			$$
			\bar{\varphi}_i^{(p)}
			=
			\frac1p\sum_{k=0}^{p-1}\varphi_i\circ f^k,
			\qquad
			\bar{\psi}_i^{(p)}\equiv1,
			\qquad
			1\leq i\leq d.
			$$
			The corresponding quotient-average map is the same as
			$\mathcal P^{(p)}$, and for every $\nu\in\mathcal M_g(X_0)$,
			$
			\mathcal P(L(\nu))=\mathcal P^{(p)}(\nu).
			$
			Hence the path can be chosen so that
			$
			\mathcal P_i(\mu_t)
			=
			t\hat{\alpha}_i+(1-t)\tilde{\alpha}_i
			$
			for every $1\leq i\leq d$ and every $t\in[0,1]$. Moreover, since
			$
			\dim_{v^{(p)}}\nu_t=\dim_v\mu_t,
			$
			the special conclusion from the mixing case also gives
			$
			\inf_{t\in[0,1]}\dim_v\mu_t
			=
			\min\{\hat h,\tilde h\}.
			$
\end{proof}

\begin{proof}[Proof of Theorem \ref{Thm-int}]
	Taking $v\equiv1$ and $\psi_i\equiv1$ for $1\le i\le d$, we have
	$
	\dim_v\mu=h_\mu(f),
	$
	$
	\dim_vK_\alpha=h_{\mathrm{top}}(f,K_\alpha).
	$
	Therefore the special assertion of Theorem \ref{Thm-int-2} gives
	Theorem \ref{Thm-int}.
\end{proof}

\section{Analytic paths for prescribed averages and Hausdorff dimension}\label{sec-hyper-dimen}

We recall some elementary facts on convex subsets of finite-dimensional Euclidean spaces, see, for example, \cite{BoydVandenberghe2004}. Let $C\subset\mathbb R^n$ be a convex set. The affine hull of $C$ is
$$
\operatorname{aff}(C)
=
\left\{
\sum_{i=1}^{k}\theta_i x_i:
k\geq 1,\ x_i\in C,\ \theta_i\in\mathbb{R},\
\sum_{i=1}^{k}\theta_i=1
\right\}.
$$
The affine dimension of $C$ is $\dim_{\mathrm{aff}}(C):=\dim\operatorname{aff}(C)$. If $C$ is convex, its relative interior is the interior of $C$ relative to $\operatorname{aff}(C)$; equivalently, $x\in\operatorname{relint}(C)$ if and only if there exists $r>0$ such that $B(x,r)\cap\operatorname{aff}(C)\subset C$. For every nonempty convex set $C$, the set $\operatorname{relint}(C)$ is nonempty and convex, and $\overline C=\overline{\operatorname{relint}(C)}$. In particular, if $C$ consists of one point, then $\operatorname{aff}(C)=\operatorname{relint}(C)=C$. If $\dim_{\mathrm{aff}}(C)=n$, then $\operatorname{relint}(C)=\operatorname{Int}(C)$, while if $\dim_{\mathrm{aff}}(C)<n$, then $\operatorname{Int}(C)=\emptyset$.

Let $C\subset\mathbb R^n$ be a convex set, and put $m:=\dim_{\mathrm{aff}}(C)$. If $m\geq1$, then there exists an index set $I_{\mathrm{aff}}\subset\{1,\ldots,n\}$ with $\sharp I_{\mathrm{aff}}=m$ such that the coordinate projection $\pi_{\mathrm{aff}}(y_1,\ldots,y_n)=(y_i)_{i\in I_{\mathrm{aff}}}$ restricts to an affine homeomorphism from $\operatorname{aff}(C)$ onto $\mathbb R^m$, and $\pi_{\mathrm{aff}}(\operatorname{relint}(C))=\operatorname{Int}(\pi_{\mathrm{aff}}(C))$. Moreover, for each $j\notin I_{\mathrm{aff}}$, there exist constants $c_{j,i},c_j^0\in\mathbb R$ such that $y_j=\sum_{i\in I_{\mathrm{aff}}}c_{j,i}y_i+c_j^0$ for all $y\in\operatorname{aff}(C)$. If $m=0$, then $\operatorname{relint}(C)=C$, and no coordinate reduction is needed.

Let $C$ be a nonempty convex subset of $\mathbb R^{m+n}$. For $x\in\mathbb R^m$, set $C_x:=\{y\in\mathbb R^n:(x,y)\in C\}$ and $C_m:=\{x\in\mathbb R^m:C_x\neq\emptyset\}$. Then
$$
\operatorname{relint}(C)
=
\{(x,y):x\in\operatorname{relint}(C_m),\ y\in\operatorname{relint}(C_x)\}.
$$

For $\Phi=(\varphi_1,\ldots,\varphi_d)\in C(X)^d$, recall that $\mathcal P(\mu)=(\int\varphi_1\,d\mu,\ldots,\int\varphi_d\,d\mu)$ for $\mu\in\mathcal M_f(X)$, and
$
K_\alpha(\Phi;f)
=
\bigcap_{i=1}^d
\left\{
x\in X:
\lim\limits_{n\to\infty}
\frac{1}{n}\sum_{k=0}^{n-1}\varphi_i(f^k x)
=
\alpha_i
\right\}.
$
This is the special case $\psi_i\equiv1$  in Theorem \ref{Thm-int-2}.
We now give a relative interior version of Theorem \ref{Thm-int-2}.
\begin{proposition}\label{prop:relint-D}
	Let $X$ be a locally maximal hyperbolic set of a $C^{1+\varepsilon}$ diffeomorphism $f$. Assume that $(X,f)$ is topologically mixing (resp. topologically transitive). Given $d\in\mathbb N^+$, $0<\gamma\leq1$, $v\in C^\gamma(X)$ with $v>0$, and $\Phi\in C^\gamma(X)^d$. Then for any $\hat\alpha,\tilde\alpha\in\operatorname{relint}\mathcal P(\mathcal M_f(X))$ and any $0<\hat h\leq\dim_vK_{\hat\alpha}$, $0<\tilde h\leq\dim_vK_{\tilde\alpha}$, there is a path
	$$
	\{\mu_t\}_{t\in[0,1]}\subset\mathcal M_f^B(X)
	\qquad
	\left(\hbox{resp. }
	\{\mu_t\}_{t\in[0,1]}\subset\mathcal M_f^e(X)\right)
	$$
	such that
	\begin{enumerate}
		\item $t\mapsto\int\psi\,d\mu_t$ is analytic for every H\"older continuous function $\psi$;
		
		\item $t\mapsto\dim_v\mu_t$ is analytic and $\inf_{t\in[0,1]}\dim_v\mu_t=\min\{\hat h,\tilde h\}$;
		
		\item $(\mathcal P(\mu_0),\dim_v\mu_0)=(\tilde\alpha,\tilde h)$ and $(\mathcal P(\mu_1),\dim_v\mu_1)=(\hat\alpha,\hat h)$;
		
		\item for every $1\leq i\leq d$ and every $t\in[0,1]$, $\mathcal P_i(\mu_t)=t\hat\alpha_i+(1-t)\tilde\alpha_i$. In particular, if $\hat\alpha=\tilde\alpha=\alpha$, then $\mathcal P(\mu_t)=\alpha$ for every $t\in[0,1]$;
		
		\item $S_{\mu_t}=X$ for every $t\in[0,1]$.
	\end{enumerate}
\end{proposition}

\begin{proof}
	Set $C:=\mathcal P(\mathcal M_f(X))\subset\mathbb R^d$. Since $\mathcal M_f(X)$ is compact and convex, and since $\mathcal P$ is continuous affine, $C$ is a compact convex subset of $\mathbb R^d$. Let $m:=\dim_{\mathrm{aff}}(C)$.
	
	If $m=0$, then $C=\{\alpha_0\}$ and hence $\hat\alpha=\tilde\alpha=\alpha_0$. Since every invariant measure has the same $\Phi$-average, then $K_{\alpha_0}=X$. The one-dimensional lifting argument used in the proof of Theorem \ref{Thm-int-2} then gives the required analytic path. This proves the case $m=0$.

	Assume now that $m\geq1$. Choose $I_{\mathrm{aff}}\subset\{1,\ldots,d\}$ with $\sharp I_{\mathrm{aff}}=m$ so that the coordinate projection $\pi_{\mathrm{aff}}(y_1,\ldots,y_d)=(y_i)_{i\in I_{\mathrm{aff}}}$ maps $\operatorname{aff}(C)$ affinely homeomorphically onto $\mathbb R^m$ and satisfies $\pi_{\mathrm{aff}}(\operatorname{relint}C)=\operatorname{Int}(\pi_{\mathrm{aff}}(C))$. For each $j\notin I_{\mathrm{aff}}$, take constants $c_{j,i},c_j^0$ such that $y_j=\sum_{i\in I_{\mathrm{aff}}}c_{j,i}y_i+c_j^0$ for all $y\in\operatorname{aff}(C)$. Put $\bar\Phi=(\varphi_i)_{i\in I_{\mathrm{aff}}}$. Then $\mathcal P^{\bar\Phi}(\mathcal M_f(X))=\pi_{\mathrm{aff}}(C)$, and for every $\alpha\in C$,
	$$
	(\mathcal P^\Phi)^{-1}(\alpha)
	=
	(\mathcal P^{\bar\Phi})^{-1}(\pi_{\mathrm{aff}}(\alpha)).
	$$
		We also have $K_\alpha(\Phi;f)=K_{\pi_{\mathrm{aff}}(\alpha)}(\bar\Phi;f)$ for every $\alpha\in C$. Indeed, for each $j\notin I_{\mathrm{aff}}$, the function $\eta_j:=\varphi_j-\sum_{i\in I_{\mathrm{aff}}}c_{j,i}\varphi_i-c_j^0$ has zero integral with respect to every invariant measure. For any $x\in X$, every accumulation point of the empirical measures $\frac1n\sum_{k=0}^{n-1}\delta_{f^kx}$ is invariant, and hence every accumulation point of the Birkhoff averages of $\eta_j$ is zero. Thus these Birkhoff averages converge to zero along every orbit. Therefore the omitted coordinates are determined by the reduced ones through the affine relations, and the two level sets coincide.

	Put $\hat a:=\pi_{\mathrm{aff}}(\hat\alpha)$ and $\tilde a:=\pi_{\mathrm{aff}}(\tilde\alpha)$. Then $\hat a,\tilde a\in\operatorname{Int}\mathcal P^{\bar\Phi}(\mathcal M_f(X))$, and  $\dim_vK_{\hat a}(\bar\Phi;f)=\dim_vK_{\hat\alpha}(\Phi;f)$ and $\dim_vK_{\tilde a}(\bar\Phi;f)=\dim_vK_{\tilde\alpha}(\Phi;f)$. Thus Theorem \ref{Thm-int-2}, in the special case with all denominators equal to one, applies to $\bar\Phi$, $\hat a$, and $\tilde a$. We obtain a path $\{\mu_t\}_{t\in[0,1]}$ satisfying the desired analytic and $v$-dimensional properties for the reduced coordinates, with $\mathcal P^{\bar\Phi}(\mu_t)=t\hat a+(1-t)\tilde a$. The affine relations among the remaining coordinates then give $\mathcal P^\Phi(\mu_t)=t\hat\alpha+(1-t)\tilde\alpha$ for all $t\in[0,1]$. The endpoint identities and the full support property are inherited from the path supplied by Theorem \ref{Thm-int-2}.
\end{proof}

Now we prove Theorem \ref{thm-huasdorff}.

\begin{proof}
	Since $X$ is a locally maximal hyperbolic set of a $C^{1+\varepsilon}$ diffeomorphism, the stable and unstable bundles are H\"older continuous; see \cite[Section 19.1]{KH1995}. Hence there exists $\gamma_0\in(0,\gamma]$ such that the functions
	$$
	\psi^u(x):=\frac{1}{\dim E^u}\log\left|\det\left(Df|_{E_x^u}\right)\right|,
	\qquad
	\psi^s(x):=\frac{1}{\dim E^s}\log\left|\det\left(Df|_{E_x^s}\right)\right|.
	$$
	belong to $C^{\gamma_0}(X)$.
	Since $\gamma_0\leq\gamma$ and $X$ is compact, $C^\gamma(X)\subset C^{\gamma_0}(X)$, and hence $$\varphi_1,\ldots,\varphi_d,\psi^u,\psi^s\in C^{\gamma_0}(X).$$ By hyperbolicity and compactness, there exists $\delta>0$ such that $\int\psi^u\,d\nu\geq\delta$ and $\int\psi^s\,d\nu\leq-\delta$ for every $\nu\in\mathcal M_f(X)$. Since $X$ is average conformal, every $\mu\in\mathcal M_f^e(X)$ satisfies
	$$
	\dim_H\mu
	=
	\frac{h_\mu(f)}{\int\psi^u\,d\mu}
	-
	\frac{h_\mu(f)}{\int\psi^s\,d\mu}.
	$$
	See \cite{WC2016} for detailed proofs, which may be viewed as an extension of Young's results in \cite{Young1982} to the average conformal setting.
	
	For $\alpha\in\mathcal P(\mathcal M_f(X))$, define $D_\alpha$ as the set of all triples $(\int\psi^u\,d\nu,\int\psi^s\,d\nu,h_\nu(f))$ with $\nu\in\mathcal M_f(X)$ and $\mathcal P(\nu)=\alpha$, and define $E_\alpha$ in the same way but with $\nu$ restricted to $\mathcal M_f^e(X)$. Also put $T(\nu):=(\int\psi^u\,d\nu,\int\psi^s\,d\nu,h_\nu(f))$ and $Q(x,y,z):=z/x-z/y$. Then $Q(T(\mu))=\dim_H\mu$ for every $\mu\in\mathcal M_f^e(X)$. Since the entropy map is affine on $\mathcal M_f(X)$, $D_\alpha$ is convex. The function $Q$ is continuous on $D_\alpha$, because the first two coordinates are uniformly separated from zero.
	
	We first prove the following auxiliary assertion. For every $\alpha\in\operatorname{relint}\mathcal P(\mathcal M_f(X))$ and every $0<h<\dim_H(\alpha)$, there exists $\mu_{\alpha,h}\in\mathcal M_f^e(X)$ such that $\mathcal P(\mu_{\alpha,h})=\alpha$, $\dim_H\mu_{\alpha,h}=h$, and $T(\mu_{\alpha,h})\in\operatorname{relint}D_\alpha$.
	
	Indeed, since $D_\alpha$ is convex, $D_\alpha\subset\overline{\operatorname{relint}D_\alpha}$. By continuity of $Q$ on $D_\alpha$, $Q(D_\alpha)\subset\overline{Q(\operatorname{relint}D_\alpha)}$. Since $\operatorname{relint}D_\alpha$ is convex, it is connected, and hence $Q(\operatorname{relint}D_\alpha)$ is an interval. Therefore $\operatorname{Int}(Q(D_\alpha))\subset Q(\operatorname{relint}D_\alpha)$. By \cite[Lemma 7.4]{DHT}, one has
	$$
	(0,\dim_H(\alpha))
	\subset
	\left\{
	\dim_H\mu:
	\mu\in\mathcal M_f^e(X),\ \mathcal P(\mu)=\alpha
	\right\}
	=
	Q(E_\alpha).
	$$
	Since $E_\alpha\subset D_\alpha$, every $h\in(0,\dim_H(\alpha))$ belongs to $\operatorname{Int}(Q(D_\alpha))$, and hence $h\in Q(\operatorname{relint}D_\alpha)$. Choose $b\in\operatorname{relint}D_\alpha$ with $Q(b)=h$. By \cite[Lemma 7.2]{DHT}, every point of $\operatorname{relint}D_\alpha$ is realized by an ergodic measure in the fiber over $\alpha$; equivalently, $\operatorname{relint}D_\alpha\subset E_\alpha$. Thus there exists $\mu_{\alpha,h}\in\mathcal M_f^e(X)$ such that $\mathcal P(\mu_{\alpha,h})=\alpha$ and $T(\mu_{\alpha,h})=b$. Consequently $\dim_H\mu_{\alpha,h}=Q(T(\mu_{\alpha,h}))=h$, proving the auxiliary assertion.
	
	Applying the auxiliary assertion to $(\hat\alpha,\hat h)$ and $(\tilde\alpha,\tilde h)$, and using $\hat\alpha,\tilde\alpha\in\operatorname{Int}\mathcal P(\mathcal M_f(X))\subset\operatorname{relint}\mathcal P(\mathcal M_f(X))$, we obtain $\hat\mu,\tilde\mu\in\mathcal M_f^e(X)$ such that $\mathcal P(\hat\mu)=\hat\alpha$, $\dim_H\hat\mu=\hat h$, $T(\hat\mu)\in\operatorname{relint}D_{\hat\alpha}$, and $\mathcal P(\tilde\mu)=\tilde\alpha$, $\dim_H\tilde\mu=\tilde h$, $T(\tilde\mu)\in\operatorname{relint}D_{\tilde\alpha}$.
	
	Set $\Phi^{\mathrm{new}}:=(\varphi_1,\ldots,\varphi_d,\psi^u,\psi^s)\in C^{\gamma_0}(X)^{d+2}$, and define $$\mathcal P^{\mathrm{new}}(\nu):=(\mathcal P(\nu),\int\psi^u\,d\nu,\int\psi^s\,d\nu).$$ Let $C^{\mathrm{new}}:=\mathcal P^{\mathrm{new}}(\mathcal M_f(X))\subset\mathbb R^{d+2}$, and let $\pi:\mathbb R^{d+2}\to\mathbb R^d$ be the projection onto the first $d$ coordinates. Then $\pi(C^{\mathrm{new}})=\mathcal P(\mathcal M_f(X))$. Let $\rho_H:\mathbb R^3\to\mathbb R^2$ be given by $\rho_H(x,y,z)=(x,y)$. For every $\alpha\in\mathcal P(\mathcal M_f(X))$, the fiber of $C^{\mathrm{new}}$ over $\alpha$ is $(C^{\mathrm{new}})_\alpha=\{y\in\mathbb R^2:(\alpha,y)\in C^{\mathrm{new}}\}=\rho_H(D_\alpha)$.
	
	Since linear maps send relative interiors of convex sets onto relative interiors of their images, $\rho_H(T(\hat\mu))\in\operatorname{relint}\rho_H(D_{\hat\alpha})=\operatorname{relint}(C^{\mathrm{new}})_{\hat\alpha}$. Also $\hat\alpha\in\operatorname{Int}\mathcal P(\mathcal M_f(X))\subset\operatorname{relint}\pi(C^{\mathrm{new}})$. Then $\hat\beta:=\mathcal P^{\mathrm{new}}(\hat\mu)=(\hat\alpha,\rho_H(T(\hat\mu)))$ belongs to $\operatorname{relint}C^{\mathrm{new}}$. The same argument gives $\tilde\beta:=\mathcal P^{\mathrm{new}}(\tilde\mu)=(\tilde\alpha,\rho_H(T(\tilde\mu)))\in\operatorname{relint}C^{\mathrm{new}}$.
	
	Put $\hat H:=h_{\hat\mu}(f)$ and $\tilde H:=h_{\tilde\mu}(f)$. The average conformal Hausdorff dimension formula gives $\hat h=\hat H(1/\int\psi^u\,d\hat\mu-1/\int\psi^s\,d\hat\mu)$, and hence $\hat H>0$. Similarly, $\tilde H>0$.
	
		We claim that the assumptions of Proposition \ref{prop:relint-D} are satisfied for $\Phi^{\mathrm{new}}$, $v\equiv1$, the endpoint averages $\hat\beta,\tilde\beta$, and the endpoint values $\hat H,\tilde H$. Let
		$$
		G_{\hat\mu}:=
		\left\{
		x\in X:
		\frac1n\sum_{k=0}^{n-1}\delta_{f^kx}\to\hat\mu
		\text{ in the weak}^*\text{ topology}
		\right\}
		$$
		be the set of $\hat\mu$-generic points. Since $\hat\mu$ is ergodic and $\mathcal P^{\mathrm{new}}(\hat\mu)=\hat\beta$, the set $G_{\hat\mu}$ is contained in $K_{\hat\beta}(\Phi^{\mathrm{new}};f)$. By Bowen's entropy formula for generic points \cite{Bowen1973TopologicalEntropyNoncompactSets}, $h_{\mathrm{top}}(f,G_{\hat\mu})=h_{\hat\mu}(f)=\hat H$. Since $v\equiv1$, this gives $\dim_vK_{\hat\beta}(\Phi^{\mathrm{new}};f)\geq\hat H$. The same argument gives $\dim_vK_{\tilde\beta}(\Phi^{\mathrm{new}};f)\geq\tilde H$. Thus $0<\hat H\leq\dim_vK_{\hat\beta}(\Phi^{\mathrm{new}};f)$ and $0<\tilde H\leq\dim_vK_{\tilde\beta}(\Phi^{\mathrm{new}};f)$.
	
	Applying Proposition \ref{prop:relint-D} to $\Phi^{\mathrm{new}}$ with $v\equiv1$, we obtain a path $\{\mu_t\}_{t\in[0,1]}$ contained in $\mathcal M_f^B(X)$ in the topologically mixing case and contained in $\mathcal M_f^e(X)$ in the topologically transitive case. This path satisfies the following properties: $t\mapsto\int\psi\,d\mu_t$ is analytic for every H\"older continuous function $\psi$; $t\mapsto h_{\mu_t}(f)$ is analytic; $S_{\mu_t}=X$ for every $t\in[0,1]$;
	$$
	(\mathcal P^{\mathrm{new}}(\mu_0),h_{\mu_0}(f))=(\tilde\beta,\tilde H),
	\qquad
	(\mathcal P^{\mathrm{new}}(\mu_1),h_{\mu_1}(f))=(\hat\beta,\hat H),
	$$
	and $\mathcal P^{\mathrm{new}}(\mu_t)=t\hat\beta+(1-t)\tilde\beta$ for all $t\in[0,1]$. Taking the first $d$ coordinates gives $\mathcal P(\mu_t)=t\hat\alpha+(1-t)\tilde\alpha$ for all $t\in[0,1]$, proving item (4). Items (1) and (5) follow directly from the construction.
	
	In the mixing case, every Bernoulli measure is ergodic. Hence, in both the mixing and transitive cases, every $\mu_t$ is ergodic, so the average conformal Hausdorff dimension formula applies to every $\mu_t$. At the endpoints, the identities for $\mathcal P^{\mathrm{new}}$ and entropy give
	$$
	\int\psi^u\,d\mu_0=\int\psi^u\,d\tilde\mu,\quad
	\int\psi^s\,d\mu_0=\int\psi^s\,d\tilde\mu,\quad
	h_{\mu_0}(f)=h_{\tilde\mu}(f),
	$$
	and
	$$
	\int\psi^u\,d\mu_1=\int\psi^u\,d\hat\mu,\quad
	\int\psi^s\,d\mu_1=\int\psi^s\,d\hat\mu,\quad
	h_{\mu_1}(f)=h_{\hat\mu}(f).
	$$
	Using the dimension formula, we obtain $\dim_H\mu_0=\dim_H\tilde\mu=\tilde h$ and $\dim_H\mu_1=\dim_H\hat\mu=\hat h$. Together with $\mathcal P(\mu_0)=\tilde\alpha$ and $\mathcal P(\mu_1)=\hat\alpha$, this proves item (3).
	
	Finally, $t\mapsto h_{\mu_t}(f)$, $t\mapsto\int\psi^u\,d\mu_t$, and $t\mapsto\int\psi^s\,d\mu_t$ are analytic, while $\int\psi^u\,d\mu_t>0$ and $\int\psi^s\,d\mu_t<0$ for every $t\in[0,1]$. Hence
	$$
	t\mapsto
	\dim_H\mu_t
	=
	\frac{h_{\mu_t}(f)}{\int\psi^u\,d\mu_t}
	-
	\frac{h_{\mu_t}(f)}{\int\psi^s\,d\mu_t}
	$$
	is analytic. This proves item (2), and hence the theorem.
\end{proof}

\section{Analytic paths for prescribed center Lyapunov exponents and entropy}\label{sec-hyper-partial}
Now we prove Theorem \ref{thm:same-sign-path}.
\begin{proof}
	We only consider the case $\hat{\chi}>0$ and $\tilde{\chi}>0$. The case $\hat{\chi}<0$ and $\tilde{\chi}<0$ follows by considering $f^{-1}$.
	
	For $\chi\in(\chi_{\min},\chi_{\max})$, recall
	$
	H(\chi)=
	\sup\{h_\mu(f):\mu\in\mathcal M_{\mathrm{erg},\chi}(f)\}.
	$
	By \cite[Theorem  1]{DiazGelfertZhang2024},  $H$ is continuous on each of the intervals
	$
	(\chi_{\min},0)
	~
	\hbox{and}
	~
	(0,\chi_{\max}).
	$
	Since
	$
		\hat h<H(\hat\chi),
		~
		\tilde h<H(\tilde\chi),
		$
		and since $\hat\chi,\tilde\chi\in(0,\chi_{\max})$, the continuity of $H$ on
		$(0,\chi_{\max})$ allows us to choose
		$
		0<\hat{\chi}^-<\hat{\chi}<\hat{\chi}^+<\chi_{\max},
		~
		0<\tilde{\chi}^-<\tilde{\chi}<\tilde{\chi}^+<\chi_{\max}
		$
		so close to $\hat\chi$ and $\tilde\chi$, respectively, that
		$$
		H(\hat\chi^\pm)>\hat h,
		\qquad
	H(\tilde\chi^\pm)>\tilde h.
	$$
	By the definition of $H$, there exist ergodic measures
	$$
	\nu_{\hat -}\in\mathcal M_{\mathrm{erg},\hat\chi^-}(f),
	\qquad
	\nu_{\hat +}\in\mathcal M_{\mathrm{erg},\hat\chi^+}(f),
	$$
	and
	$$
	\nu_{\tilde -}\in\mathcal M_{\mathrm{erg},\tilde\chi^-}(f),
	\qquad
	\nu_{\tilde +}\in\mathcal M_{\mathrm{erg},\tilde\chi^+}(f)
	$$
	such that
	$
	h_{\nu_{\hat -}}(f)>\hat h,
	~
	h_{\nu_{\hat +}}(f)>\hat h,
	$
	and
	$
	h_{\nu_{\tilde -}}(f)>\tilde h,
	~
	h_{\nu_{\tilde +}}(f)>\tilde h.
	$
	Since all four center exponents are positive, these measures are hyperbolic of
	center-expanding type. By
	\cite[Remark~4.4, Definition~4.5 and Lemma~4.7]{DiazGelfertZhang2024}, the manifold $M$
	has the $(d^{uu}+1)$-ergodic approximation property for center-expanding
	hyperbolic measures.
	Choose $\varepsilon>0$ so small that
	$$
	2\varepsilon<
	\min\{
	h_{\nu_{\hat -}}(f)-\hat h,\,
	h_{\nu_{\hat +}}(f)-\hat h,\,
	h_{\nu_{\tilde -}}(f)-\tilde h,\,
	h_{\nu_{\tilde +}}(f)-\tilde h
	\}
	$$
	and
	$$
	\varepsilon<
	\min\{
	\hat\chi-\hat\chi^-,\,
	\hat\chi^+-\hat\chi,\,
	\tilde\chi-\tilde\chi^-,\,
	\tilde\chi^+-\tilde\chi
	\}.
	$$ Applying this approximation property
	to the four measures
	$
	\nu_{\hat -},~ \nu_{\hat +},~
	\nu_{\tilde -},~ \nu_{\tilde +},
	$
	we obtain topologically transitive locally maximal hyperbolic sets
	$
	\Lambda_{\hat -},~ \Lambda_{\hat +},~
	\Lambda_{\tilde -},~ \Lambda_{\tilde +}
	$
	of center expanding type such that, for each corresponding pair
	$(\nu,\Lambda_\nu)$,
	$$
	h_{\mathrm{top}}(f,\Lambda_\nu)\geq h_\nu(f)-\varepsilon
	$$
	and
	$$
	|\chi^c(\mu)-\chi^c(\nu)|<\varepsilon
	\quad
	\hbox{for every }
	\mu\in\mathcal M_{\mathrm{erg}}(f|_{\Lambda_\nu}).
	$$
	For each $\Lambda_\nu$, choose an ergodic measure $\eta_\nu$ such that
	$
	h_{\eta_\nu}(f)>h_{\mathrm{top}}(f,\Lambda_\nu)-\varepsilon.
	$
	By the choice of $\varepsilon$, we get
	$$
	\chi^c(\eta_{\hat -})<\hat\chi<\chi^c(\eta_{\hat +}),
	\qquad
	h_{\eta_{\hat -}}(f)>\hat h,\quad h_{\eta_{\hat +}}(f)>\hat h,
	$$
	and
	$$
	\chi^c(\eta_{\tilde -})<\tilde\chi<\chi^c(\eta_{\tilde +}),
	\qquad
	h_{\eta_{\tilde -}}(f)>\tilde h,\quad
	h_{\eta_{\tilde +}}(f)>\tilde h.
	$$
	
	By \cite[Remark~4.4]{DiazGelfertZhang2024}, any two saddles with the same
	$u$-index are homoclinically related. Since
	$
	\Lambda_{\hat -},~ \Lambda_{\hat +},~
	\Lambda_{\tilde -},~ \Lambda_{\tilde +}
	$
	are basic sets of center-expanding type, all their periodic points have
	$u$-index $d^{uu}+1$. Hence the periodic points in any two of these basic sets are homoclinically related.
	We now use the same transition argument as in the proof of
	\cite[Lemma~4.13]{DiazGelfertZhang2024}. Namely, whenever two topologically transitive locally maximal hyperbolic sets of the same
	$u$-index have homoclinically related periodic points, they can be embedded in a common topologically transitive locally maximal hyperbolic set of that same $u$-index. Applying this argument successively to the four hyperbolic sets above, we obtain a topologically transitive locally maximal hyperbolic set
	$
	\Lambda
	$
	of $u$-index $d^{uu}+1$ such that
	$
	\Lambda_{\hat -}\cup\Lambda_{\hat +}\cup
	\Lambda_{\tilde -}\cup\Lambda_{\tilde +}
	\subset \Lambda.
	$
	 Since its $u$-index is $d^{uu}+1$, it is of center-expanding type.

		Let
		$$
		\varphi^c(x):=\log\|Df|_{E^c(x)}\|.
		$$
		On $\Lambda$, the partially hyperbolic splitting refines the unstable bundle as
		$
		E^u_\Lambda=E^c\oplus E^{uu}.
		$
		This is an invariant dominated splitting. Hence, by
		\cite[Theorem~A.15]{BochiPotrieSambarino2019}, after possibly decreasing the
		H\"older exponent, the bundle $E^c|_\Lambda$ is H\"older continuous.
		Since $f$ is $C^{1+\varepsilon}$, it follows that $\varphi^c|_\Lambda$ is
		H\"older continuous.
	
	Since
	$
	\chi^c(\eta_{\hat -})<\hat\chi<\chi^c(\eta_{\hat +}),
	$
	there exists $\lambda_{\hat\chi}\in(0,1)$ such that
	$
	\lambda_{\hat\chi}\chi^c(\eta_{\hat -})
	+
	(1-\lambda_{\hat\chi})\chi^c(\eta_{\hat +})
	=
	\hat\chi.
	$
	Set
	$
	\rho_{\hat\chi}:=
	\lambda_{\hat\chi}\eta_{\hat -}
	+
	(1-\lambda_{\hat\chi})\eta_{\hat +}.
	$
	Then $\rho_{\hat\chi}\in\mathcal M_f(\Lambda)$ and
	$
	\int \varphi^c\,d\rho_{\hat\chi}=\hat\chi.
	$
	Moreover, 
	$
	h_{\rho_{\hat\chi}}(f)
	=
	\lambda_{\hat\chi}h_{\eta_{\hat -}}(f)
	+
	(1-\lambda_{\hat\chi})h_{\eta_{\hat +}}(f)
	>
	\hat h.
	$
	Therefore
	$$
	\hat h<
	\sup\left\{
	h_\mu(f):
	\mu\in\mathcal M_f(\Lambda),\
	\int \varphi^c\,d\mu=\hat\chi
	\right\}.
	$$
	Similarly,
	$$
	\tilde h<
	\sup\left\{
	h_\mu(f):
	\mu\in\mathcal M_f(\Lambda),\
	\int \varphi^c\,d\mu=\tilde\chi
	\right\}.
	$$
	Moreover, the inequalities
	$
	\chi^c(\eta_{\hat -})<\hat\chi<\chi^c(\eta_{\hat +}),
	~
	\chi^c(\eta_{\tilde -})<\tilde\chi<\chi^c(\eta_{\tilde +})
	$
	show that
	$
	\hat\chi,\tilde\chi
	\in
	\operatorname{Int}
	\left\{
	\int_\Lambda \varphi^c\,d\mu:
	\mu\in\mathcal M_{f|_\Lambda}(\Lambda)
	\right\}.
	$
	For $a\in\mathbb R$, set
	$$
	K_a^\Lambda:=
	\left\{
	x\in\Lambda:
	\lim_{n\to\infty}
	\frac{1}{n}\sum_{j=0}^{n-1}\varphi^c(f^j x)=a
	\right\}.
	$$
	For the system $(\Lambda,f|_\Lambda)$, apply Lemma~\ref{lemma-analytic}
	with
	$$
	d=1,\qquad
	\Phi=\varphi^c|_\Lambda,\qquad
	\Psi=1,\qquad
	v\equiv 1.
	$$
	Then, for every
	$
	a\in
	\operatorname{Int}
	\left\{
	\int_\Lambda \varphi^c\,d\mu:
	\mu\in\mathcal M_{f|_\Lambda}(\Lambda)
	\right\},
	$
	we have
	$$
	\htop(f|_\Lambda,K_a^\Lambda)
	=
	\max
	\left\{
	h_\mu(f|_\Lambda):
	\mu\in\mathcal M_{f|_\Lambda}(\Lambda),\
	\int_\Lambda \varphi^c\,d\mu=a
	\right\}.
	$$
	Since 
	$
	h_\mu(f|_\Lambda)=h_\mu(f),
	$
	the preceding inequalities imply
	$
	0<\hat h<\htop(f|_\Lambda,K_{\hat\chi}^\Lambda),
	~
	0<\tilde h<\htop(f|_\Lambda,K_{\tilde\chi}^\Lambda).
	$
	Thus the hypotheses of Theorem~\ref{Thm-int} applied to
	$(\Lambda,f|_\Lambda)$ and the H\"older potential $\varphi^c|_\Lambda$ are satisfied.
	Therefore there exists a path
	$$
	\{\mu_t\}_{t\in [0,1]}\subset \mathcal{M}_{f|_\Lambda}^e(\Lambda)\subset \mathcal{M}_f^e(M)
	$$
	such that:
		\begin{enumerate}
		\item the map
		$
		t\mapsto \int \psi\, d\mu_t
		$
		is analytic for any H\"older continuous function $\psi:\Lambda\to \mathbb{R}$;
		\item the map
		$
		t\mapsto h_{\mu_t}(f)
		$
		is analytic and
		$
		\inf_{t\in [0,1]} h_{\mu_t}(f)=\min\{\hat{h},\tilde{h}\};
		$
			\item
			$
			\left(\int_\Lambda\varphi^c\,d\mu_0,h_{\mu_0}(f)\right)=(\tilde{\chi},\tilde{h}),
			$ and
			$
			\left(\int_\Lambda\varphi^c\,d\mu_1,h_{\mu_1}(f)\right)=(\hat{\chi},\hat{h});
			$
			\item
			$
			\int_\Lambda\varphi^c\,d\mu_t=t\hat{\chi}+(1-t)\tilde{\chi} 
			$ for any $t\in[0,1]$.
		\end{enumerate}
		Since every $\mu_t$ is ergodic and supported on $\Lambda$, we have
		$
		h_{\mu_t}(f|_\Lambda)=h_{\mu_t}(f),
		$
		and
		$
		\chi^c(\mu_t)=\int_\Lambda \varphi^c\,d\mu_t.
		$
	Moreover, every H\"older continuous function on $M$ restricts to a H\"older
	continuous function on $\Lambda$. Hence the analytic properties obtained on
	$\Lambda$ give the corresponding analytic properties for the original system
	$(M,f)$.
	
	The case $\hat\chi<0$ and $\tilde\chi<0$ follows by applying the preceding
	argument to $f^{-1}$. Indeed, entropy is unchanged, center Lyapunov exponents
	change sign, and the stable and unstable blender assumptions as well as the
	minimality of the strong stable and unstable foliations are interchanged under
	passing to the inverse map.
\end{proof}

\subsection*{Acknowledgments}
X. Hou is supported by the National Natural Science Foundation of China No. 12401231 and the Fundamental Research Funds for the Central Universities No. DUT25RC(3)106. X. Tian is supported by the National Natural Science Foundation of China (No. 12471182) and Natural Science Foundation of Shanghai (No. 23ZR1405800).

 \noindent\textbf{Conflict of interest.} The authors declare that there is no conflict of interest.
 
 \noindent\textbf{Data availability.} No data was used for the research described in the article.

 \phantomsection
 \addcontentsline{toc}{section}{References}
 \bibliographystyle{plain}

\end{document}